\documentclass[12pt]{article}

\usepackage{amsmath, amssymb}
\usepackage{mathrsfs}   
\usepackage{geometry,graphicx}
\usepackage{tikz}
\usetikzlibrary{calc}
\usepackage{url}
\usepackage{enumerate}

\usepackage{amsthm}

\theoremstyle{plain}
\newtheorem{theorem}{Theorem}[section]
\newtheorem{conjecture}[theorem]{Conjecture}
\newtheorem{lemma}[theorem]{Lemma}
\newtheorem{proposition}[theorem]{Proposition}
\newtheorem{corollary}[theorem]{Corollary}

\newtheorem{question}[theorem]{Question}
\newtheorem{notation}[theorem]{Notation}

\theoremstyle{definition}
\newtheorem{definition}[theorem]{Definition}

\theoremstyle{remark}
\newtheorem{remark}[theorem]{Remark}
\newtheorem{example}[theorem]{Example}

\usepackage{mathtools}

\newcommand{\R}{\mathbb{R}}
\newcommand{\Fc}{\mathcal{F}}
\newcommand{\conv}{\operatorname{conv}}

\newcommand{\joincs}{\mathbin{\#}}
\newcommand{\pface}{\mathsf{\Pi}}
\newcommand{\cs}{\#_{\pface}}
\newcommand{\uset}{\operatorname{set}}

\DeclareMathOperator{\ri}{rint} 
\DeclareMathOperator{\aff}{aff}

\DeclarePairedDelimiter\floor{\lfloor}{\rfloor}
\DeclarePairedDelimiter\set{\{}{\}}
\DeclarePairedDelimiter\paren{(}{)}

\newcommand{\f}{\varphi}

\usepackage{xcolor}

\title{Equifacetted polytopes whose facets are products of simplices: The simplicial prism case}
\author{Federico Castillo, Guillermo Pineda-Villavicencio, José Alejandro Samper}
\date{\today}

\begin{document}

\maketitle

\begin{abstract}
We study \emph{prismatical} polytopes, those all of whose facets are combinatorially isomorphic to the simplicial prism $\Delta^{d-2}\times\Delta^1$.
In dimension three these are exactly the cubical polytopes, so the class sits naturally between simplicial and cubical polytopes.
In dimension four, $\Delta^2\times\Delta^2$ is the simplest prismatical polytope and we can construct infinitely many by using connected sums.
By fibreing a triangle over a triangulated polygon we construct infinitely many connected-sum-prime prismatical $4$-polytopes.
Finally, we show that there are no prismatical $d$-polytopes for $d\ge5$, answering a question of Perles and Shephard.
\end{abstract}

\section{Introduction}

A recurring theme in polytope theory is the study of  \emph{equifacetted} polytopes, that is, polytopes whose facets are all combinatorially isomorphic.
Two classical and well-studied classes are \emph{simplicial polytopes} \cite[\S4.5]{Gru03}, in which every facet is a simplex, and \emph{cubical polytopes} \cite[\S4.6]{Gru03}, in which every facet is a cube. 
In both cases, the strong combinatorial constraints imposed at the level of facets lead to rich structural theories and deep connections with topology, geometry, and combinatorics.

In this note we investigate a class of polytopes that naturally interpolates between these two settings.
Namely, we study polytopes whose facets are all combinatorially isomorphic to prisms over simplices.
We call these \textbf{prismatical polytopes}.
We prove their nonexistence in dimensions greater than four, describe the cases with few facets in dimension four,
and construct infinite families of connected-sum-prime examples in dimension four.

Perles and Shephard~\cite{PerShe67} asked which polytopes occur as the common facet type of an equifacetted polytope one dimension higher, calling them \emph{facets} and the remaining ones \emph{nonfacets}, and asked whether simple nonfacets exist in every dimension at least four~\cite[Question (iv), \S 8]{PerShe67}.
Shortly after, Barnette~\cite{Bar69} constructed an infinite family of simple $4$-dimensional nonfacets.
To our knowledge, no simple nonfacets in dimension five or higher have previously been reported.
In this language, Theorem~\ref{thm:product-facet-1-characterization}(viii) says that the simplicial prism $\Delta^{d-1}\times\Delta^1$ is a nonfacet for every $d\ge4$, so it gives a simple nonfacet in every dimension at least four.
Another consequence of our results is that, for every $d\ge4$, the smallest number of facets of a $d$-dimensional nonfacet is $d+2$: the simplicial prism attains this bound, whereas  a $d$-polytope with $d+1$ facets is a simplex and hence a facet. Unlike the connected-sum construction of Perles and Shephard~\cite[Theorem~(9)]{PerShe67}, ours gives infinitely many connected-sum-prime 4-dimensional examples with simplicial prisms as facets.

We now describe our main contributions.
Two (convex) polytopes are combinatorially isomorphic if their face lattices are isomorphic.
In this paper, we do not distinguish between combinatorially isomorphic polytopes. 

Throughout, a \emph{$d$-polytope} means a convex polytope of dimension $d$.
We write $\Delta^k$ for the $k$-simplex.

\begin{definition}[Simplicial prisms and prismatical polytopes]
Let $d\geq2$.
A \emph{simplicial prism of dimension $d-1$} is a polytope combinatorially isomorphic to $\Delta^{d-2}\times\Delta^1$.
We usually write \emph{prism} for \emph{simplicial prism}.

A $d$-polytope $\mathsf P$ is \emph{prismatical} if every facet of $\mathsf P$ is a simplicial prism.
We denote the class of prismatical $d$-polytopes by $\mathscr P_d$.
\end{definition}

\begin{remark}
For $d=3$, the polytope $\Delta^{d-2}\times \Delta^1$ is the square, hence $\mathscr{P}_3$ coincides with the class of \emph{cubical} $3$-polytopes (all facets are squares).
For $d=2$ the family $\mathscr{P}_2$ equals the set of all polygons.
\end{remark}

The 4-polytope $\Delta^2\times\Delta ^2$ is prismatical.
Further examples can be obtained by taking connected sums.

\begin{definition}[Connected sum along facets]
\label{def:connected-sum}
Let $\mathsf{P}$ and $\mathsf{P}'$ be $d$-polytopes, and let $\mathsf F\subset\mathsf P$ and $\mathsf F'\subset\mathsf P'$ be projectively isomorphic facets.
Apply suitable projective transformations to the two polytopes so that $\mathsf F=\mathsf F'$, the two polytopes lie in opposite closed halfspaces bounded by $\aff\mathsf F$, and their union is convex, with facets precisely the facets of $\mathsf P$ and $\mathsf P'$ other than $\mathsf F$ and $\mathsf F'$.
The resulting convex $d$-polytope is called a \emph{connected sum} of $\mathsf P$ and $\mathsf P'$ along $\mathsf F$ and $\mathsf F'$, and is denoted by $ \mathsf P\mathbin{\#_{\mathsf F}}\mathsf P'$.
\end{definition}

A polytope $\mathsf{P}$ is \emph{projectively unique} if every polytope combinatorially isomorphic to $\mathsf{P}$ is projectively isomorphic to $\mathsf{P}$.  
Every $d$-polytope with $d+2$ vertices or facets is projectively unique \cite[Ex.~2.15.19]{Pin24}.
Since a  simplicial prism of dimension $d$ has $d+2$ facets, it is projectively unique, so the gluing facets can always be identified projectively.

A prismatical $4$-polytope is \textbf{connected-sum-prime} if it admits no decomposition as a connected sum (along a triangular-prism facet).
The polytope $\Delta^2\times \Delta^2$ is connected-sum-prime.

Since every prismatical $4$-polytope has an even number of facets (Lemma~\ref{lem:dim4-basic}), every such polytope has at least six facets.
This bound is attained by $\Delta^2\times\Delta^2$, which is remarkable in several respects.
We prove the following structural result.

\begin{theorem}[Structure and constructions of prismatical polytopes]\label{thm:product-facet-1-characterization}
Let $d\ge 2$ and let $\mathscr{P}_d$ be the class of prismatical $d$-polytopes.
\begin{enumerate}[{\rm (i)}]
  \item If $d=2$, then $\mathscr{P}_2$ is the class of all polygons.
  \item If $d=3$, then $\mathscr{P}_3$ is exactly the class of cubical $3$-polytopes, for which a substantial theory is known, see \cite{BliBli98, JosZie00, SchZie04}.
  \item Up to combinatorial isomorphism, $\Delta^2\times\Delta^2$ is the unique simple prismatical polytope of dimension at least $4$.
  It is also the unique prismatical $4$-polytope with six facets.
  
\item There is no prismatical  $4$-polytope with eight facets. 

\item Every prismatical 4-polytope with ten facets is combinatorially isomorphic to $(\Delta^2\times \Delta^2)\#(\Delta^2\times \Delta^2)$.
In particular, there is no connected-sum-prime example.

\item Every prismatical 4-polytope with fourteen facets is combinatorially isomorphic to a connected sum of three copies of  $(\Delta^2\times \Delta^2)$.
In particular, there is no connected-sum-prime example.

  
\item For each even number $n\in \mathbb{N} \setminus \left\{ 2,4,6,8,10,14 \right\}$, there is a connected-sum-prime, prismatical $4$-polytope with $n$ facets. 

\item If $d>4$, then $\mathscr{P}_d$ is empty.
\end{enumerate}
\end{theorem}


The construction in Theorem~\ref{thm:product-facet-1-characterization}(vii) is a polytopal fibre bundle, which allows us to express the combinatorics of a four dimensional polytope using a two dimensional picture.
This is equivalent to a construction of Pfeifle, Pilaud and Santos~\cite[Section~2.2]{PfePilSan12}.
Alternatively, after the projective preprocessing as described in \cite[Sec.~2.5]{Pin24}, it may be viewed as the image under a linear projection of a $5$-polytope $\widetilde{\mathsf P}(\mathsf X)$ combinatorially isomorphic to $\Delta^2\times\mathsf X$, where $\mathsf X$ is an auxiliary 3-polytope.
The Projection Lemma of Sanyal and Ziegler~\cite[Lem.~2.5]{SanZie10} shows that this projection strictly preserves the  $1$-skeleton of $\widetilde{\mathsf P}(\mathsf X)$. 

We end the introduction with a bold conjecture, which would generalise Theorem~\ref{thm:product-facet-1-characterization}(viii).

\begin{conjecture} Let $p,q$ be integers satisfying $p\ge q\ge 1$ and $p+q=d-1$.
If $p>q+1$, then there is no  $d$-polytope whose facets are all combinatorially equivalent to $\Delta^p \times \Delta^q$.
    \label{conj:product-facet-nonexistence}
\end{conjecture}

If the sizes of the factors are not fixed, there are many polytopes whose facets are all products of simplices.
Let $\mathsf P$ be a simplicial $d$-polytope and let $H$ be a hyperplane that meets the interior of $\mathsf P$ but contains none of its vertices.
Every facet of the $(d-1)$-polytope $\mathsf P\cap H$ has the form $\mathsf F\cap H$, where $\mathsf F$ is a facet of $\mathsf P$ whose vertices are not all on the same side of $H$.
If $a$ vertices of the $(d-1)$-simplex $\mathsf F$ lie on one side of $H$ and $b=d-a$ on the other, then $\mathsf F\cap H$ is combinatorially isomorphic to $\Delta^{a-1}\times\Delta^{b-1}$, the product of the two simplices spanned by these vertex sets.
Thus every facet of $\mathsf P\cap H$ is a product of two simplices, but in general the pair $(a,b)$ changes from facet to facet.

\section{Preliminaries}
\label{sec:preliminaries}

We begin by recording the combinatorial data used to specify a connected sum.
We denote by $\mathcal L(\mathsf P)$ the \emph{face lattice}  of a polytope $\mathsf P$.

\begin{definition}
\label{def:gluing-map}
Let $\mathsf{P}$ and $\mathsf{P'}$ be polytopes with facets $\mathsf{F}$ and $\mathsf{F'}$ that are projectively isomorphic. The projective identification of $\mathsf F$ with $\mathsf F'$ induces a combinatorial isomorphism \[
\f\colon\mathcal L(\mathsf F)\longrightarrow\mathcal L(\mathsf F'),
\]
called the \emph{gluing map}.
Thus, for every proper face $\mathsf G\subsetneq\mathsf F$, the face $\mathsf G$ is identified with $\f(\mathsf G)\subsetneq\mathsf F'$ in the connected sum.
\end{definition}

Definitions~\ref{def:connected-sum} and~\ref{def:gluing-map}
immediately give the face numbers of a connected sum.
As usual, $f_k(\mathsf P)$ denotes the number of $k$-faces of $\mathsf P$.

\begin{remark}
    \label{rem:connected-sum-faces} Let $\mathsf{P}$ and $\mathsf{P}'$ be $d$-polytopes with two respective facets  $\mathsf F$ and $\mathsf F'$ projectively isomorphic.
The number of proper faces of $\mathsf P\joincs_{\mathsf F} \mathsf P'$ is as follows.
\begin{align*}
 f_k(\mathsf P\joincs_{\mathsf F} \mathsf P')
   &=f_k(\mathsf P)+f_k(\mathsf P')-f_k(\mathsf F),
       &&0\leq k\leq d-2\\
 f_{d-1}(\mathsf P\joincs_{\mathsf F} \mathsf P')
   &=f_{d-1}(\mathsf P)+f_{d-1}(\mathsf P')-2.
\end{align*}
\end{remark}

\begin{proposition}[Closure under connected sum]
\label{prop:connected-sum-closure}
The class $\mathscr P_d$ is closed under connected sums.
\end{proposition}

\begin{proof}
Any two simplicial prisms of the same dimension are projectively isomorphic.
All facets of $\mathsf{P}\#\mathsf{Q}$ are inherited from $\mathsf{P}$ or $\mathsf{Q}$
except for the two facets used in the gluing, which are removed.
Hence every facet of $\mathsf{P}\#\mathsf{Q}$ is still combinatorially isomorphic to $\Delta^{d-2}\times \Delta^1$.
\end{proof}

For a face $\mathsf F$ of  $\mathsf P$, set 
 \begin{equation*}
   \Fc_k(\mathsf F,\mathsf P):=\set{\mathsf G\in\mathcal L(\mathsf P): \mathsf F\subseteq\mathsf G,\ \dim\mathsf G=k}, 
	 \qquad \text{and} \qquad f_k(\mathsf F,\mathsf P):=\#\Fc_k(\mathsf F,\mathsf P). 
\end{equation*}
Thus $\Fc_k(\mathsf F,\mathsf P)$ is the set of $k$-faces of $\mathsf P$ containing $\mathsf F$, and $f_k(\mathsf F,\mathsf P)$ is its cardinality.
In particular, $\Fc_k(\mathsf P):=\Fc_k(\varnothing,\mathsf P)$ and $ f_k(\mathsf P):=f_k(\varnothing,\mathsf P)$ denote, respectively, the set and the number of $k$-faces of $\mathsf P$.
When $\mathsf P$ is fixed, we abbreviate $f_k(\mathsf P)$ to $f_k$.

For a polytope $\mathsf P$, let $\mathcal B(\mathsf P)$ denote its \emph{boundary complex}, consisting of all proper faces of $\mathsf P$.
For a finite pure $d$-dimensional polytopal complex $\mathcal C$, its \emph{boundary complex} $\mathcal B(\mathcal C)$ is the subcomplex generated by the $(d-1)$-faces contained in exactly one facet of $\mathcal C$.
The \emph{underlying set} of  $\mathcal C$, denoted $\uset  \mathcal C$, is the set of points in  the ambient space  that belongs to at least one polytope in $\mathcal C$.
The $k$-skeleton of $\mathsf P$ is the subcomplex formed by its faces of dimension at most $k$.
The \emph{graph} $G(\mathsf P)$ of $\mathsf P$ is its $1$-skeleton.
The \emph{dual graph} $G^*(\mathsf P)$ has the facets of $\mathsf P$ as its vertices, with two facets adjacent precisely when they intersect in a ridge.
See \cite[Sec.~2.11]{Pin24} for definitions not provided here.

In a $d$-polytope, a vertex is \emph{simple} if $f_{d-1}(u,\mathsf P)=d$; otherwise it is \emph{nonsimple}.
Similarly, an edge $\mathsf e$ is \emph{simple} if $f_{d-1}(\mathsf e,\mathsf P)=d-1$ and \emph{nonsimple} if $f_{d-1}(\mathsf e,\mathsf P)\ge d$.

We continue with some basic properties of the dual of a prismatical polytope.

\begin{lemma} Let $\mathsf P$ be a prismatical $d$-polytope, where $d\ge4$, and let $\mathsf  P^*$ be the dual polytope of $\mathsf P$.
The following hold. 
\begin{enumerate}[\rm (i)]
    \item  Every vertex figure of $\mathsf P^*$  is a bipyramid over a $(d-2)$-simplex.
    
   \item  Every vertex of $\mathsf P^*$ has degree $d+1$.  
   \item Every facet of $\mathsf P^*$ is a simple polytope. 
\end{enumerate}
        \label{lem:prismatical-dual-basics}
\end{lemma}
\begin{proof}
(i)-(ii)     For each vertex $v$ of the dual polytope $\mathsf P^*$, the vertex figure $\mathsf  P^*/v$ is dual to the facet of $\mathsf  P$ associated with $v$, which is a simplicial $(d-1)$-prism.
Thus $\mathsf  P^*/v$ is combinatorially isomorphic to a bipyramid over a $(d-2)$-simplex.
It follows that $\mathsf  P^*/v$ has $d+1$ vertices.
Hence every vertex of $\mathsf P^*$ has degree $d+1$.
This settles (i) and (ii).

(iii)   If $\mathsf F$ is a facet of $\mathsf  P^*$ and $v$ is a vertex of $\mathsf F$ , then the vertex figure $\mathsf  F/v$ is a facet of $\mathsf  P^*/v$; it is therefore a simplex.
Thus every facet  of $\mathsf P^*$ is a simple polytope. \end{proof}

We next characterise the  simple prismatical polytopes. 
For this we need the following classical structural results from McMullen and Barnette.

\begin{enumerate}
    \item Let  $\mathsf P$ be a $d$-polytope with $d+2$ facets, where $d\ge 2$.  Then there exist integers $2\le a\le d$ and $1\le m\le \floor{a/2}$ such that $\mathsf P$ is a $(d-a)$-fold pyramid over $\Delta ^m\times \Delta^{a-m}$, see \cite[Sec.~3]{McM71}.  \label{lem:dplus2facets}
    
    \item Let $d\ge 4$  and let $\mathsf P$ be a simplicial $d$-polytope with $n$ vertices. Then $f_2(P) \geq \binom{d}{2}n-\binom{d+1}{3}2$. This is part of the more general Barnette's lower bound Theorem, see \cite{Bar71,Bar73}. \label{thm:simple-lbt} 
    
\end{enumerate}



\begin{proposition}\label{prop:smallest-prismatical}
Let $\mathsf P$ be a prismatical $d$-polytope, where $d\ge4$. 
If $\mathsf P$ is simple or $\mathsf{P}$ has six facets, then $d=4$ and \[
\mathsf P\cong\Delta^2\times\Delta^2.\]
\end{proposition}

\begin{proof}
Let  $\mathsf{P}$ be a simple prismatical $d$-polytope.  
By Lemma~\ref{lem:dim4-basic}, a prismatical $4$-polytope has an even number of facets, so it has at least six. 

By Lemma~\ref{lem:prismatical-dual-basics}, every vertex of  the dual $\mathsf P^*$ has degree $d+1$. 
Let $n=f_0(\mathsf P^*)$ and $e=f_1(\mathsf P^*)$. 
Then $e=(d+1)n/2$. 
By Barnette's Lower Bound Theorem for edges of simplicial $d$-polytopes, $e\ge dn-\binom{d+1}{2}$.
Hence $d(d+1)\ge n(d-1)$.
In other words, 
\begin{equation}
    n\le d+2+2/(d-1)
\end{equation}
This gives $d+2 \geq n$ for $d\ge 4$.
Thus  $\mathsf P$ has at most $d+2$ facets. 
A $d$-polytope with $d+1$ facets is a simplex and hence is not prismatical. 
The simple $d$-polytopes with $d+2$ facets are Cartesian products of two simplices $\Delta^m\times \Delta^{d-m}$, where $1\le m\le d/2$.
If $m=1$, then $\mathsf P$ has a simplex facet, which is impossible.
If $m\geq3$, then the facet $\Delta^{m-1}\times\Delta^{d-m}$ is not a simplicial prism. 
Hence $m=2$. 

The facet $\Delta^2\times\Delta^{d-3}$ is then a simplicial prism only when $d=4$. 
This settles the uniqueness of $\Delta ^2\times\Delta ^2$ among the simple prismatical polytopes.
If $ \mathsf{P} $ has six facets, then its dimension is four or five.
Five cannot be as it would be a simplex which is not prismatical.
Thus $d=4$ and $\mathsf P$ has six facets. The pyramid cases in McMullen's classification are excluded. Hence $\mathsf P$ is simple, and the preceding argument gives $\mathsf P\cong\Delta^2\times\Delta^2$.
\end{proof}

\section{Dimension four}
\label{sec:prismatical-dim-4}
In dimension four we have the central example of the product of two triangles.
\begin{notation}
We write $\mathsf B:=\Delta^2\times\Delta^2$ and $\pface:=\Delta^2\times\Delta^1$.
The polytope $\mathsf B$ will be called the \emph{bitriangle}.
\end{notation}

By Proposition~\ref{prop:connected-sum-closure}, infinitely many prismatical 4-polytopes can be obtained via repeated connected sums having $\Delta^2\times \Delta^2$ as a summand.
As we will see, prismatical polytopes with few facets are often connected sums of copies of $ \mathsf{B} $.
However, we will construct infinitely many prismatical polytopes that are prime, in the sense that they are not the connected sum of two other prismatical polytopes.

Our first result deals with basic counts in a prismatical 4-polytope. 
For a vertex $u$ in a  polytope $\mathsf P$, we write $\deg_{\mathsf P}(u):=\deg_{G(\mathsf P)}(u)$.

\begin{lemma}\label{lem:dim4-basic}
Suppose $\mathsf P$ is a prismatical 4-polytope. 
Then $f_3$ is even,
\begin{equation}\label{eq:counts}
  f_2=\frac{5f_3}2,\qquad f_1=f_0+\frac{3f_3}2,
  \qquad \sum_{u\in \Fc_0(\mathsf P)} f_3(u,\mathsf P)=6f_3.
\end{equation}
For every vertex $u$, $f_3(u,\mathsf P)$ is even, with \begin{equation}\label{eq:loc}
  \deg_{\mathsf P}(u)=\frac12 f_3(u,\mathsf P)+2,
  \qquad 3f_3(u,\mathsf P)\le2f_0-6 .
\end{equation}

Additionally,  if $f_0\le16$, then $f_3(u,\mathsf P)\in\{4,6,8\}$, and the possible vertex figures $\mathsf P/u$, with the corresponding values of $f_3(\mathsf e,\mathsf P)$ on their incident edges, are \begin{equation}\label{eq:figs}
\begin{array}{llll}
 f_3(u,\mathsf P)=4:&\text{tetrahedron},&(3,3,3,3);&\\
 f_3(u,\mathsf P)=6:&\text{triangular bipyramid},&(4,4,4,3,3);&\\
 f_3(u,\mathsf P)=8:&\text{octahedron},&(4,4,4,4,4,4);&\\
 f_3(u,\mathsf P)=8:&\text{stacked type},&(5,5,4,4,3,3).&
\end{array}
\end{equation}
\end{lemma}

\begin{proof}
Each prism has two triangular and three quadrilateral ridges, and each ridge lies in exactly two facets. 
By counting the facet-ridge incidences in $\mathsf P$, we get  $2f_2=5f_3$. 
Thus $f_3$ is even. 
Euler--Poincar\'e--Schl\"afli's equation for 4-polytopes (see \cite[Thm.~2.12.17]{Pin24}) then gives $f_1=f_0+f_2-f_3=f_0+3f_3/2.$

To finish (\ref{eq:counts}), note that each prism has six vertices, and each vertex $u$ in $\mathsf P$ is contained in exactly $f_3(u,\mathsf P)$ facets, so counting facet-vertex incidences yields 
\[
\sum_{u\in \Fc_0(\mathsf P)} f_3(u,\mathsf P)=6f_3.
\] 

To prove (\ref{eq:loc}), fix a vertex $u$ and put $p:=f_3(u,\mathsf P)$. 
Since a prism is a simple polytope, every facet of $\mathsf P/u$ is a triangle, so $\mathsf P/u$ is a simplicial $3$-polytope. 
Put $\widehat f_i=f_i(\mathsf P/u)$. 
Then $\widehat f_2=p$, and facet--edge incidences in $\mathsf P/u$ give  $3\widehat f_2=2\widehat f_1$. 
Consequently,  $p$ is even, and by Euler--Poincar\'e--Schl\"afli's equation for $\mathsf P/u$, namely $\widehat f_0-\widehat f_1+\widehat f_2=2$,  
 \[\deg_{\mathsf P}(u)=\widehat f_0=p/2+2.\]
 
The $p$ prism facets containing $u$ give $5p$ pairs 
\[
 (\mathsf F,w),
 \qquad \mathsf F\in\Fc_3(\mathsf P),\quad u,w\in\mathsf F,\quad w\ne u.
\]
Write $\bar v\in \mathsf P/u$ for the vertex  corresponding to an edge $[u,v]\in \mathsf P$.  
Since $f_3([u,v],\mathsf P)=\deg_{\mathsf P/u}(\bar v)$, the pairs for which $w$ is adjacent to $u$ account for 
\begin{equation*}
\sum_{\substack{v\in\Fc_0(\mathsf P)\\
                  [u,v]\in\Fc_1(\mathsf P)}}
       f_3([u,v],\mathsf P)=\sum_{\bar v\in \Fc_0(\mathsf P/u)}\deg_{\mathsf P/u}(\bar v)= 2\widehat f_1=3p.
\end{equation*}
The remaining $2p$ pairs have $w$ nonadjacent to $u$. 
For every such vertex $w$, the vertices $u$ and $w$ are opposite vertices of a quadrilateral ridge. 
This ridge is unique and is contained in exactly two facets of $\mathsf P$. 
Consequently, each cofacial nonneighbour $w$ of $u$ occurs in exactly two pairs. 
Let $c(u)$ be the number of nonneighbours of $u$ that lie with $u$ in a facet. 
It follows that $2p=2c(u)$.
The total number of nonneighbours of $u$ is $f_0-1-\deg_{\mathsf P}(u) =f_0-3-p/2$. 
Therefore $ p=c(u)\le f_0-3- p/2$, or equivalently 
\[
 3p\le 2f_0-6.
\]

If $f_0\le16$, then $3p\le26$, so $p\in\{4,6,8\}$; the first two figures are forced. 
For $p=8$, up to combinatorial equivalence,  there are exactly two simplicial $3$-polytopes $\mathsf P/u$, both with six vertices:
the octahedron, whose vertices all have degree four, and the connected sum of a tetrahedron and a triangular bipyramid along triangular facets, whose degree sequence is $(5,5,4,4,3,3)$.  
The numbers in \eqref{eq:figs} are precisely the vertex degrees of $\mathsf P/u$.
\end{proof}

We shall use the following lemma which guarantees that if we have two simple vertices in a triangle face, then the polytope can be decomposed.

\begin{lemma}\label{lem:split}
Let $\mathsf P$ be a prismatical 4-polytope with $f_3(\mathsf P)\ge10$.
If a triangular ridge of $\mathsf P$ has at least two simple vertices, then $\mathsf P\cong(\Delta^2\times\Delta^2)\cs\mathsf Q$, where $\mathsf Q$ is a prismatical 4-polytope satisfying \[
 f_3(\mathsf Q)=f_3(\mathsf P)-4,
 \qquad f_0(\mathsf P)=f_0(\mathsf Q)+3.
\]
\end{lemma}
\begin{proof}
We use the labels of Figure~\ref{fig:B-BcsB-subgraphK}(a).
Let $\mathsf T=[b_1,b_2,b_3]$ be a triangular ridge of $\mathsf P$ such that $b_1$ and $b_2$ are simple vertices, and let $\mathsf F^{+}$ and $\mathsf F^{-}$ be the two facets of $\mathsf P$ containing $\mathsf T$.
Since $\mathsf F^{+}$ is a prism with $\mathsf T$ as one of its triangles, its other triangle has vertices $a_1,a_2,a_3$, labeled so that $a_ib_i$ is an edge for each $i$.
Similarly, the other triangle of $\mathsf F^{-}$ has vertices $c_1,c_2,c_3$, with $c_ib_i$ an edge for each $i$.

Since $b_1$ is simple, it has exactly four neighbours, namely $a_1$, $c_1$, $b_2$ and $b_3$, and any three of the edges at $b_1$ lie in a common facet.
Let $\mathsf G$ be the facet containing the edges $a_1b_1$, $b_1c_1$ and $b_1b_2$; it is a prism by assumption.
The quadrilaterals $[a_1,a_2,b_2,b_1]$ of $\mathsf F^{+}$ and $[b_1,b_2,c_2,c_1]$ of $\mathsf F^{-}$ are the $2$-faces at $b_1$ spanned by two of these edges, so they are faces of $\mathsf G$ as well.
A prism is determined by two of its quadrilaterals sharing an edge, so the triangles of $\mathsf G$ are $[a_1,b_1,c_1]$ and $[a_2,b_2,c_2]$.
The same argument applied at $b_1$ with $b_3$ in place of $b_2$, and at $b_2$ with $b_3$ in place of $b_1$, yields three facets $\mathsf G_1$, $\mathsf G_2$ and $\mathsf G_3$ with vertex sets $\Fc_0(\mathsf G_1)=\set{a_1,a_2,b_1,b_2,c_1,c_2}$, $\Fc_0(\mathsf G_2)=\set{a_1,a_3,b_1,b_3,c_1,c_3}$ and $\Fc_0(\mathsf G_3)=\set{a_2,a_3,b_2,b_3,c_2,c_3}$.

The vertex $b_3$ is simple as well.
It lies in the four prisms $\mathsf F^{+}$, $\mathsf F^{-}$, $\mathsf G_2$ and $\mathsf G_3$, in each of which it has exactly three neighbours, and each three of the edges $b_3a_3$, $b_3c_3$, $b_3b_1$, $b_3b_2$ lie in one of these prisms.
Hence the four corresponding triangles of the vertex figure of $\mathsf P$ at $b_3$ form the boundary of a tetrahedron, so the vertex figure is a tetrahedron.
In particular, each $b_i$ lies in exactly four facets, and the facets of $\mathsf P$ containing some $b_i$ are precisely $\mathsf F^{+}$, $\mathsf F^{-}$, $\mathsf G_1$, $\mathsf G_2$ and $\mathsf G_3$.

Finally we show that $\mathsf P$ is a connected sum of $\mathsf B'=\operatorname{conv}\set{a_i,b_i,c_i : i=1,2,3}$ and $\mathsf Q=\operatorname{conv}\big(\Fc_0(\mathsf P)\setminus\set{b_1,b_2,b_3}\big)$ along $\mathsf R=\operatorname{conv}\set{a_1,a_2,a_3,c_1,c_2,c_3}$.
The triangles $[a_1,a_2,a_3]$ and $[c_1,c_2,c_3]$ are faces of $\mathsf F^{+}$ and $\mathsf F^{-}$, and for $i<j$ the quadrilateral $[a_i,a_j,c_j,c_i]$ is a face of the prism $\mathsf G_k$ containing both $b_i$ and $b_j$.
In particular, any two of the lines $a_1c_1$, $a_2c_2$, $a_3c_3$ are coplanar, so the three lines are coplanar, concurrent or parallel, and in every case the six points lie in a $3$-dimensional affine subspace.
Hence $\mathsf R$ is a triangular prism, and its boundary is the union of these five faces of $\mathsf P$.

Put $H:=\aff\mathsf R$. 
Then $\mathsf P\cap H=\mathsf R$. 
Since $\mathsf T$ is  disjoint from $\mathsf R$, it lies in one open halfspace bounded by $H$. 
Every neighbour of $\mathsf T$ outside $\mathsf T$ lies in $H$, so every vertex outside $\mathsf T\cup\mathsf R$ lies in the opposite open halfspace. 
Hence $H$ supports $\mathsf Q$ at $\mathsf R$. 
Since $\mathsf P$ has at least ten vertices, $\mathsf Q$ has a vertex outside $H$; thus $\mathsf Q$ is $4$-dimensional and $\mathsf R$ is a facet of it.

Then $\mathsf B'$ is bounded by $\mathsf F^{+}$, $\mathsf F^{-}$, $\mathsf G_1$, $\mathsf G_2$, $\mathsf G_3$ and $\mathsf R$, so $\mathsf B'\cong\Delta^2\times\Delta^2$.
The facets of $\mathsf Q$ are $\mathsf R$ together with the facets of $\mathsf P$ that contain no $b_i$, so $\mathsf Q$ is prismatical, $f_3(\mathsf Q)=f_3(\mathsf P)-5+1=f_3(\mathsf P)-4$, and $f_0(\mathsf Q)=f_0(\mathsf P)-3$.
\end{proof}

\subsection{Small number of facets}

We shall see that when the number of facets is small, a prismatical 4-polytope is a connected sum of copies of $ \mathsf{B} $.
However first we show a nonexistence result.

\begin{proposition}
    \label{prop:no-eight}There is no prismatical $4$-polytope with exactly eight facets.
\end{proposition}

\begin{proof} Assume that such a polytope $\mathsf P$ exists, and let $\mathsf Q=\mathsf P^\ast$ be its dual polytope.  
Every facet of $\mathsf Q$ is a simple $3$-polytope (Lemma~\ref{lem:prismatical-dual-basics}), and thus has an even number of vertices.
A facet of the eight-vertex polytope $\mathsf Q$ has at most seven vertices.
Consequently, every facet of $\mathsf Q$ has either four or six vertices.

We claim that a six-vertex facet is impossible.
Suppose that $\mathsf F$
has six vertices, leaving exactly two vertices $x,y$ outside $\mathsf F$.
Each vertex of $\mathsf F$ has three neighbours in $\mathsf F$ and degree five in $\mathsf Q$ (Lemma~\ref{lem:prismatical-dual-basics}), and hence must be adjacent to both $x$ and $y$.
It follows that $x$ is adjacent to all six vertices of $\mathsf F$,
contradicting $\deg_{\mathsf Q}(x)=5$.

Every facet of $\mathsf Q$ therefore has four vertices, so $\mathsf Q$ is a simplicial $4$-polytope and $\mathsf P$ is a simple polytope.
Proposition~\ref{prop:smallest-prismatical} gives that $\mathsf P$ is $\Delta^2 \times \Delta^2$.
This contradiction shows that  $\mathsf P$ cannot exist.
\end{proof}

Next, we deal with prismatical 4-polytopes with 10 facets. 

\begin{theorem} \label{thm:no-prime-ten}
Every prismatical 4-polytope with ten facets is combinatorially isomorphic to a connected sum of two copies of $\Delta^2\times\Delta^2$. 
In particular, there is no connected-sum-prime example.
\end{theorem}

\begin{proof}
Here $f_3=10$, so $\sum_{u\in \Fc_0(\mathsf P)} f_3(u,\mathsf P)=60$ by \eqref{eq:counts}.
Since $f_3(u,\mathsf P)\ge4$, we have $f_0\le15$.
Lemma~\ref{lem:dim4-basic} therefore gives $f_3(u,\mathsf P)\in\{4,6,8\}$.
If some vertex were incident with 8 facets, then \eqref{eq:loc} would give $24\le2f_0-6$, and hence $f_0=15$.
But then $\sum_{u\in \Fc_0(\mathsf P)} f_3(u,\mathsf P)=60=4f_0$, so every  vertex would be simple, contradicting Proposition~\ref{prop:smallest-prismatical}.
Hence every vertex of $\mathsf P$ is incident with either $4$ or $6$ facets. 


Put $n_i:=\#\{u\in\Fc_0(\mathsf P):f_3(u,\mathsf P)=i\}$. 
By Proposition~\ref{prop:smallest-prismatical}, $\mathsf P$ is not simple, so some vertex is incident with six facets. 
Lemma~\ref{lem:dim4-basic} gives $18\le2f_0-6$, hence $f_0\ge12$. 
Since $4n_4+6n_6=60$  (see \eqref{eq:counts}) and $n_4+n_6=f_0$, we obtain $n_4=3f_0-30\ge6$. 
Each simple vertex belongs to exactly two triangular ridges. 
Likewise, the ten prism facets give exactly ten triangular ridges. 
Thus there are $2n_4\ge12>10$ incidences between simple vertices and triangular ridges. 
By the pigeonhole principle, some triangular ridge contains at least two simple vertices. 
Lemma~\ref{lem:split}, together with Proposition~\ref{prop:smallest-prismatical}, therefore shows that $\mathsf P$ is the connected sum of two copies of $\mathsf B$.

\end{proof}

\begin{remark}\label{rmk:unique-ten}
The connected sum of two copies of $\mathsf B=\Delta^2\times\Delta^2$ has a unique combinatorial type. Automorphisms of $\mathsf B$ act transitively on its facets, and every automorphism of a prism facet extends to $\mathsf B$ by permuting the vertices of the two triangle factors.
\end{remark}

It turns out that a similar statement is true for prismatical $4$-polytopes with fourteen facets. 
The proof is an involved case-by-case analysis and it is included as an appendix.

\subsection{Infinite families of prime prismatical 4-polytopes}
\label{subsec:polytope-families}

The purpose of this section is to construct connected-sum-prime prismatical 4-polytopes.
%
The idea is to consider a \emph{polytopal fibre bundle} over a polytope.
This same construction appeared in \cite{PfePilSan12} in the context of graphs of polytopes.
Here we present it from a different point of view.
We first develop this idea in general and then specialise to the four dimensional case.

\begin{definition}\label{def:fibration}
Let $\mathsf{R}\subset\mathbb{R}^{m}$ and $\mathsf{S}\subset\mathbb{R}^{n}$ be full-dimensional polytopes, called the \textbf{fibre} and the \textbf{base}, respectively.
We require that the fibre contains the origin in its interior (see Example~\ref{ex:fibre_origin} below), and for notational convenience we require the same of the base.
Let $\psi\colon\mathsf{S}\to\mathbb{R}$ be a concave piecewise linear function that is identically one on the boundary of $\mathsf{S}$; by concavity, $\psi\geq 1$ on all of $\mathsf{S}$.
The \textbf{total space} is
\begin{equation}\label{eq:total_space}
	\mathsf{T} \coloneqq \left\{ (x,y) \in \mathbb{R}^{m} \times \mathbb{R}^{n} ~:~ y\in \mathsf{S},\ x\in \psi(y) \mathsf{R} \right\},
\end{equation}
and the projection $\mathsf{T}\to\mathsf{S}$ onto the second factor is called the \textbf{fibration map}.
\end{definition}

By construction, the fibre of the fibration map over a point $y\in\mathsf{S}$ is the dilate $\psi(y)\mathsf{R}$.
In particular, if $\psi$ is identically one, the total space is the Cartesian product $\mathsf{R}\times\mathsf{S}$.

\begin{example}\label{ex:fibre_origin}
Consider fibrations of a segment over a segment, as in Figure~\ref{fig:segment_over_segment}.
When $\psi\equiv 1$ the total space is a rectangle, the product of the two segments.
If instead $\psi$ equals one at the endpoints of the base and two at its midpoint, the total space is a hexagon, provided the fibre contains the origin in its interior.
Otherwise the set defined by Equation~\eqref{eq:total_space} is not even convex.
For a three-dimensional example, Figure~\ref{fig:triangle_over_segment} shows the fibration of a triangle over a segment.
\end{example}

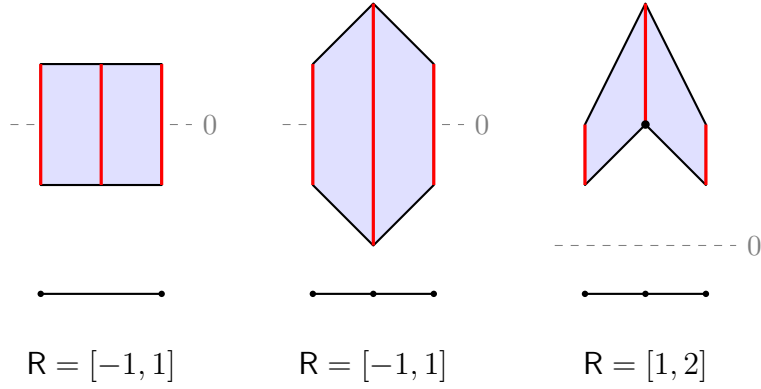
\begin{figure}[ht]
\centering
\begin{tikzpicture}[scale=0.8, line join=round]
  \begin{scope}
    \draw[dashed, gray] (-0.5,0) -- (2.5,0) node[right] {\small $0$};
    \filldraw[fill=blue!12, draw=black, thick]
      (0,-1) -- (2,-1) -- (2,1) -- (0,1) -- cycle;
    \foreach \t in {0, 1, 2}
      \draw[very thick, red] (\t,-1) -- (\t,1);
    \draw[thick] (0,-2.8) -- (2,-2.8);
    \foreach \t in {0, 1}
      \fill (2*\t,-2.8) circle (1.5pt) ;
    \node at (1,-4) {$\mathsf{R}=[-1,1]$};
  \end{scope}

  \begin{scope}[xshift=4.5cm]
    \draw[dashed, gray] (-0.5,0) -- (2.5,0) node[right] {\small $0$};
    \filldraw[fill=blue!12, draw=black, thick]
      (0,-1) -- (1,-2) -- (2,-1) -- (2,1) -- (1,2) -- (0,1) -- cycle;
    \foreach \t/\h in {0/1, 1/2, 2/1}
      \draw[very thick, red] (\t,-\h) -- (\t,\h);
    \draw[thick] (0,-2.8) -- (2,-2.8);
    \foreach \t/\v in {0/1, 1/2, 2/1}
      \fill (\t,-2.8) circle (1.5pt);
    \node at (1,-4) {$\mathsf{R}=[-1,1]$};
  \end{scope}

  \begin{scope}[xshift=9cm, yshift=-2cm]
    \draw[dashed, gray] (-0.5,0) -- (2.5,0) node[right] {\small $0$};
    \filldraw[fill=blue!12, draw=black, thick]
      (0,1) -- (1,2) -- (2,1) -- (2,2) -- (1,4) -- (0,2) -- cycle;
    \foreach \t/\h in {0/1, 1/2, 2/1}
      \draw[very thick, red] (\t,\h) -- (\t,2*\h);
    \fill (1,2) circle (2pt);   
    \draw[thick] (0,-0.8) -- (2,-0.8);
    \foreach \t/\v in {0/1, 1/2, 2/1}
      \fill (\t,-0.8) circle (1.5pt);
    \node at (1,-2) {$\mathsf{R}=[1,2]$};
  \end{scope}
\end{tikzpicture}
\caption{Fibrations of a segment over a segment. Left: the trivial fibration, $\psi\equiv 1$. Center and right: $\psi$ equals $1$ at the endpoints of the base and $2$ at its midpoint. In the center the fibre contains the origin in its interior; on the right it does not, and the total space fails to be convex.}
\label{fig:segment_over_segment}
\end{figure}

\begin{figure}[ht]
\centering
\begin{tikzpicture}[x={(1.6cm,0cm)}, y={(0.48cm,0.36cm)}, z={(0cm,0.8cm)},
                    line join=round]
  \coordinate (a0) at (0, 0, 1);
  \coordinate (b0) at (0,-1,-1);
  \coordinate (c0) at (0, 1,-1);
  \coordinate (a1) at (1, 0, 2);
  \coordinate (b1) at (1,-2,-2);
  \coordinate (c1) at (1, 2,-2);
  \coordinate (a2) at (2, 0, 1);
  \coordinate (b2) at (2,-1,-1);
  \coordinate (c2) at (2, 1,-1);

  \fill[blue!10] (a0)--(b0)--(b1)--(a1)--cycle;
  \fill[blue!10] (a1)--(b1)--(b2)--(a2)--cycle;
  \fill[blue!20] (a2)--(b2)--(c2)--cycle;

  \draw[dotted, gray] (0,0,0) -- (2,0,0);

  \draw[dashed] (c0)--(c1)--(c2);
  \draw[dashed, red!70] (b0)--(c0)--(a0);
  \draw[dashed, red!70] (b1)--(c1)--(a1);

  \draw[thick] (a0)--(a1)--(a2);
  \draw[thick] (b0)--(b1)--(b2);
  \draw[very thick, red] (a0)--(b0);
  \draw[very thick, red] (a1)--(b1);
  \draw[very thick, red] (a2)--(b2)--(c2)--cycle;

  \draw[thick] (0,0,-4) -- (2,0,-4);
  \foreach \t/\v in {0/1, 1/2, 2/1}
    \fill (\t,0,-4) circle (1.5pt) node[below] {\small $\psi=\v$};
\end{tikzpicture}
\caption{A triangle fibreed over a segment, with $\psi$ equal to $1$ at the
endpoints of the base and $2$ at its midpoint. The fibres over the endpoints
and the midpoint are drawn in red. The total space has two triangular facets
and six trapezoidal facets.}
\label{fig:triangle_over_segment}
\end{figure}
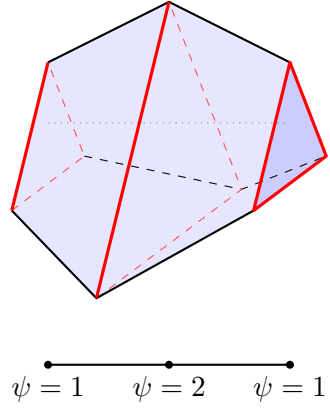

We now show that under the hypotheses of Definition~\ref{def:fibration} the total space is a polytope, and we describe its vertices and facets.
The proof also gives the vertex--facet incidences, which determine the whole face lattice.

The function $\psi$ induces a polyhedral subdivision $\mathcal{P}$ of the base whose maximal cells are the maximal regions on which $\psi$ is affine.
Since $\psi$ is constant on the boundary of $\mathsf{S}$, every proper face of $\mathsf{S}$ is a face of $\mathcal{P}$.
We write $\mathcal{C}$ for the set of maximal cells of $\mathcal{P}$ and $\mathcal{V}$ for its set of vertices.

\begin{theorem}[see also Proposition 2.7 of \cite{PfePilSan12}]\label{thm:total_space}
Under the hypotheses of Definition~\ref{def:fibration}, the total space $\mathsf{T}$ is an $(m+n)$-dimensional polytope.
Its vertices are the points in the set
\begin{equation}\label{eq:total_vertices}
	\left\{  (\psi(v)r, v) ~~:~ v\in \mathcal{V},\ r\in \mathcal{F}_{0}(\mathsf{R})\right\}
\end{equation}
Its facets are of two types:
\begin{enumerate}[{\rm (i)}]
	\item for each facet $\mathsf{F}$ of $\mathsf{S}$, the product $\mathsf{R}\times\mathsf{F}$, which is the preimage of $\mathsf{F}$ under the fibration map;
	\item for each cell $\mathsf{C}\in\mathcal{C}$ and each facet $\mathsf{G}$ of $\mathsf{R}$, the set of pairs $(x,y)$ with $y\in\mathsf{C}$ and $x\in\psi(y)\mathsf{G}$, which is combinatorially equivalent to $\mathsf{G}\times\mathsf{C}$.
\end{enumerate}
\end{theorem}

\begin{proof}
Since the fibre and the base contain the origin in their interiors and $\psi$ is concave and piecewise linear, there are finite sets $\mathrm{A}$ and $\mathrm{B}$ of linear functionals on $\mathbb{R}^m$ and $\mathbb{R}^n$, respectively, and a finite set $\mathrm{L}$ of affine functionals on $\mathbb{R}^n$, such that
\begin{align*}
	\mathsf{R} &= \left\{ x \in \mathbb{R}^{m} ~:~ \alpha(x) \leq 1 \text{ for all } \alpha \in \mathrm{A} \right\},\\
	\mathsf{S} &= \left\{ y \in \mathbb{R}^{n} ~:~ \beta(y) \leq 1 \text{ for all } \beta \in \mathrm{B} \right\},\\
	\psi(y) &= \min \left\{ \ell(y) ~:~ \ell \in \mathrm{L} \right\}.
\end{align*}
We take these descriptions to be irredundant: each $\alpha\in\mathrm{A}$ attains equality on a facet $\mathsf{G}_\alpha$ of $\mathsf{R}$, each $\beta\in\mathrm{B}$ on a facet $\mathsf{F}_\beta$ of $\mathsf{S}$, and each $\ell\in\mathrm{L}$ agrees with $\psi$ exactly on a cell $\mathsf{C}_\ell\in\mathcal{C}$, distinct functionals giving distinct cells.

With this notation we can rewrite the conditions in Equation \eqref{eq:total_space} as follows:
\begin{equation}
	\mathsf{T} = \left\{ (x,y) \in \mathbb{R}^{m} \times \mathbb{R}^{n} ~:~ \beta(y)\leq 1,\ \alpha(x) \leq \ell(y),\ (\alpha, \beta, \ell) \in \mathrm{A} \times \mathrm{B} \times \mathrm{L}\right\}.
\end{equation}
This is a finite system of linear inequalities, and its solution set is bounded because it is contained in $(\max\psi)\mathsf{R}\times\mathsf{S}$, so $\mathsf{T}$ is a polytope.
It is full-dimensional, since $\psi\geq 1$ implies $\mathsf{R}\times\mathsf{S}\subseteq\mathsf{T}$.

We now determine the face on which each inequality holds with equality, which shows that the description is irredundant and gives the facets.
There are two types of inequalities.
\begin{enumerate}
  \item First consider $ \beta(y) = 1 $.
Since $\psi=1$ on $\mathsf{F}_\beta$, the face of $\mathsf{T}$ where $\beta(y)=1$ is $\mathsf{R}\times\mathsf{F}_\beta$, of dimension $m+n-1$.
\item Now we consider $ \alpha(x) = \ell(y) $. This forces $ \psi(y) = \ell(y) $, which happens when $ y \in \mathsf{C}_{\ell} $, and $ x \in \psi(y) \mathsf{G}_{\alpha} $.
	The set of equality is the \emph{lift} of the cell $ \mathsf{C}_{\ell} $ along the facet $ \mathsf{G}_{\alpha} $ of the fibre; this lift is combinatorially  isomorphic to $\mathsf G_\alpha\times\mathsf C_\ell$.
	This is also $ m+n-1 $ dimensional.
\end{enumerate}

The vertices of $ \mathsf{T} $ are the union of the vertices of each of its facets.
From the above description it follows that the vertices are precisely the points in Equation \eqref{eq:total_vertices}.
\end{proof}

We now specialise to dimension four.

\begin{corollary}\label{cor:prismatical}
Let the fibre $\mathsf{R}$ be a triangle and the base $\mathsf{S}$ a polygon, and suppose that every maximal region of linearity of $\psi$ is a triangle, that is $\mathcal{P}$ is a triangulation of $\mathsf{S}$ (with interior vertices).
Then the total space $\mathsf{T}$ is a prismatical $4$-polytope.
\end{corollary}

\begin{proof}
By Theorem~\ref{thm:total_space}, each facet of $\mathsf{T}$ is either the product $\mathsf{R}\times\mathsf{F}$ of the triangle $\mathsf{R}$ with an edge $\mathsf{F}$ of $\mathsf{S}$, or is combinatorially equivalent to the product $\mathsf{G}\times\mathsf{C}$ of an edge $\mathsf{G}$ of $\mathsf{R}$ with a cell $\mathsf{C}\in\mathcal{C}$, which is a triangle by assumption.
In both cases the facet is a triangular prism.
\end{proof}

\begin{example}
We can fibre a triangle over a hexagon with a piecewise linear function as in Figure \ref{fig:triangulated_hexagon}. 
The total space for the fibration is a prismatical polytope with $10 \times 3 + 6 = 36$ facets.
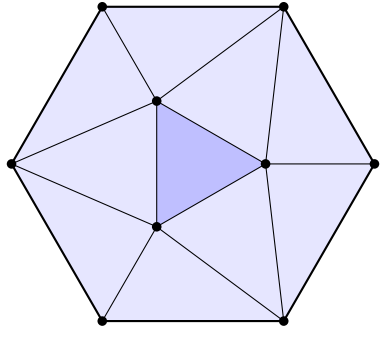
\begin{figure}[ht]
\centering
\begin{tikzpicture}[scale=1.2, line join=round]
  \foreach \i in {0,...,5}
    \coordinate (v\i) at (60*\i:2);
  \foreach \j in {0,1,2}
    \coordinate (p\j) at (120*\j:0.8);

  \fill[blue!10] (v0)--(v1)--(v2)--(v3)--(v4)--(v5)--cycle;
  \fill[blue!25] (p0)--(p1)--(p2)--cycle;

  \foreach \j in {0,1,2}
    \foreach \k in {-1,0,1} {
      \pgfmathtruncatemacro{\m}{mod(2*\j+\k+6,6)}
      \draw (p\j) -- (v\m);
    }
  \draw (p0)--(p1)--(p2)--cycle;

  \draw[thick] (v0)--(v1)--(v2)--(v3)--(v4)--(v5)--cycle;

  \foreach \i in {0,...,5} \fill (v\i) circle (1.5pt);
  \foreach \j in {0,1,2}   \fill (p\j) circle (1.5pt);
\end{tikzpicture}
\caption{A regular triangulation of a hexagon. Taking $\psi=1$ on the boundary of the
hexagon and $\psi=2$ on the inner triangle, its domains of linearity are
exactly these ten triangles.}
\label{fig:triangulated_hexagon}
\end{figure}
\end{example}

Borrowing notation from \cite{PfePilSan12} for $v\in\mathcal{V}$ and a vertex $r$ of $\mathsf{R}$ we write $\rho(r,v) \coloneqq(\psi(v)r,v)$ for the corresponding vertex of $\mathsf{T}$.

\begin{lemma}\label{lem:graph}
Let the fibre $\mathsf{R}$ be a triangle and the base $\mathsf{S}$ a polygon.
Two vertices $\rho(r, v)$ and $\rho(r', w)$ of $\mathsf{T}$ span an edge if and only if either $v=w$ and $r\neq r'$, or $r=r'$ and $vw$ is an edge of $\mathcal{P}$.
In other words, the graph of $\mathsf{T}$ is the Cartesian product of the graph of $\mathcal{P}$ with $K_3$.
\end{lemma}

\begin{proof}
The set of edges is the union of all edges of all facets,
so the statement follows from the description of the facets in Theorem \ref{thm:total_space}.
\end{proof}

\begin{lemma}\label{lem:prism_subgraphs}
Let $W$ be a set of six vertices of $\mathsf{T}$ inducing a subgraph isomorphic to the graph of a triangular prism.
Then $W$ is of one of the following forms:
\begin{enumerate}[{\rm (i)}]
	\item $\{\rho(r, v), \rho(r, w) : r \text{ a vertex of } \mathsf{R}\}$ for an edge $vw$ of $\mathcal{P}$;
	\item $\{\rho(r, u), \rho(r, v), \rho(r, w), \rho(r', u), \rho(r', v), \rho(r', w)\}$ for a $3$-cycle $uvw$ in the graph of $\mathcal{P}$ and two distinct vertices $r, r'$ of $\mathsf{R}$.
	\item $\{\rho(r, u) : u\in U\}$ for a fixed vertex $r$ of $\mathsf{R}$ and a set $U$ of six vertices of $\mathcal{P}$ which induce the graph of a triangular prism;
\end{enumerate}
\end{lemma}

\begin{proof}
A triangle in the graph of $\mathsf{T}$ either lies in a single fibre, and is then $\{\rho(r, v) : r \text{ a vertex of } \mathsf{R}\}$ for some $v$, or has the form $\{\rho(r, u),\rho(r, v),\rho(r, w)\}$ for a fixed $r$ and a $3$-cycle $uvw$.
Indeed, two edges of a triangle that change the vertex of $\mathsf{R}$ must both stay in one fibre, and then so does the third.
The graph of a prism consists of two disjoint triangles joined by a perfect matching.

\begin{itemize}
  \item If both triangles are fibres over $v$ and $w$, the matching edges leave the fibre, so $vw$ is an edge of $\mathcal{P}$; this is case (i).
  \item If a fibre over $v$ is matched with a triangle $\{\rho(r, u),\rho(r,u'),\rho(r,u'')\}$ not containing $\rho(r, v)$, then a vertex $\rho(r', v)$ with $r'\neq r$ would need a neighbour $\rho(r, x)$ with $x\neq v$, which is impossible.
	\item If both triangles are $\{\rho(r, u),\rho(r, v),\rho(r, w)\}$ and $\{\rho(r', u'),\rho(r', v'),\rho(r', w')\}$, then we consider two cases:
		\begin{itemize}
			\item $ \mathbf{r \neq r '} $ and the matching condition forces $\{u',v',w'\}=\{u,v,w\}$ which is case (ii). 
			\item $ \mathbf{r = r '} $ and we must have case (iii). \qedhere
		\end{itemize}
\end{itemize}
\end{proof}

The third case in Lemma \ref{lem:prism_subgraphs} can happen as Figure \ref{fig:induced_prism} shows.
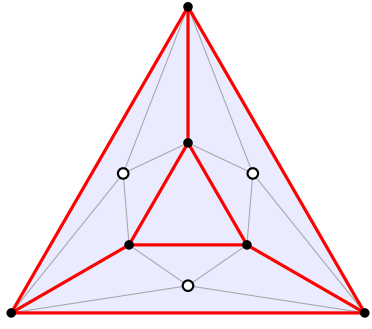
\begin{figure}[ht]
\centering
\begin{tikzpicture}[scale=0.9, line join=round]
  \foreach \i/\ang in {0/90, 1/210, 2/330} {
    \coordinate (a\i) at (\ang:3);
    \coordinate (b\i) at (\ang:1);
  }
  \foreach \i/\ang in {0/150, 1/270, 2/30}
    \coordinate (c\i) at (\ang:1.1);

  \fill[blue!8] (a0)--(a1)--(a2)--cycle;

  \begin{scope}[gray!70]
    \foreach \i in {0,1,2} {
      \pgfmathtruncatemacro{\j}{mod(\i+1,3)}
      \draw (c\i) -- (a\i);
      \draw (c\i) -- (a\j);
      \draw (c\i) -- (b\i);
      \draw (c\i) -- (b\j);
    }
  \end{scope}

  \begin{scope}[very thick, red]
    \draw (a0)--(a1)--(a2)--cycle;
    \draw (b0)--(b1)--(b2)--cycle;
    \foreach \i in {0,1,2}
      \draw (a\i)--(b\i);
  \end{scope}

  \foreach \i in {0,1,2} {
    \fill (a\i) circle (2pt);
    \fill (b\i) circle (2pt);
    \filldraw[fill=white, draw=black, thick] (c\i) circle (2.2pt);
  }
\end{tikzpicture}
\caption{A regular triangulation of a triangle whose graph contains an induced prism, drawn in red. }
\label{fig:induced_prism}
\end{figure}

\begin{theorem}\label{thm:prime}
In the setting of Corollary~\ref{cor:prismatical}, the prismatical polytope $\mathsf{T}$ is not a connected sum of two prismatical polytopes.
\end{theorem}

\begin{proof}
Suppose $\mathsf{T}$ is a connected sum along a triangular prism $\mathsf{Q}$.
Then there is a hyperplane $H$ such that $H\cap\mathsf{T}=\mathsf{Q}$, the vertices of $\mathsf{Q}$ are vertices of $\mathsf{T}$, and both open sides of $H$ contain vertices of $\mathsf{T}$.
In particular $H$ is not a supporting hyperplane, and it does not meet any edge of $\mathsf{T}$ in a single relative-interior point, since such a point would be a vertex of $\mathsf{Q}$.
Moreover, every edge of $\mathsf{T}$ joining two vertices of $\mathsf{Q}$ lies in $H\cap\mathsf{T}=\mathsf{Q}$, so the vertices of $\mathsf{Q}$ induce a prism graph and Lemma~\ref{lem:prism_subgraphs} applies.
In each case the six vertices span a $3$-dimensional affine subspace, which must be $H$.
We analise how the hyperplane $H$ intersects the polytope in each of the cases of Lemma~\ref{lem:prism_subgraphs}.

\begin{description}
    \item[Case (i)]
We have $H=\mathbb{R}^2\times L$, where $L$ is the line through the vertices $v$ and $w$ of $\mathcal{P}$.
If $vw$ is an edge of $\mathsf{S}$, then $H$ supports the facet $\mathsf{R}\times[v,w]$.
Otherwise one endpoint, say $v$, is in the interior of $\mathsf{S}$: if both were on the boundary, $\psi$ would equal one at an interior point of $\mathsf{S}$, forcing the concave function $\psi$ to be constant.
Then $L\cap\mathsf{S}$ strictly contains $[v,w]$, and $H\cap\mathsf{T}$, the preimage of $L\cap\mathsf{S}$, is strictly larger than $\mathsf{Q}$.
    \item[Case (ii)]
Let $\alpha$ be the linear functional equal to one on $\mathsf{G}=[r,r']$, and $\ell$ the affine function agreeing with $\psi$ at $u,v,w$; then $H=\{\alpha(x)=\ell(y)\}$.
If $uvw$ is a cell, $H$ supports a facet.
Otherwise, since $\mathcal{P}$ is a triangulation, some vertex $z\in\mathcal{V}$ lies in the interior of the triangle $uvw$.
By concavity, $\psi(z)>\ell(z)$: equality $\psi(z)=\ell(z)$ would force $\psi=\ell$ throughout $uvw$, making $uvw$ a cell of $\mathcal P$.
Let $q$ be the third vertex of $\mathsf{R}$; since the origin is in the interior of $\mathsf{R}$, we have $\alpha(q)<0<\ell(z)$.
Then $H$ strictly separates $\rho(r, z)$, where $\alpha(\psi(z)r)=\psi(z)>\ell(z)$, from the vertex $\rho(q,z)$, where $\alpha(\psi(z)q)<\ell(z)$.
So $H$ meets the edge $\rho(r, z)\rho(q,z)$ in a single relative-interior point.
    \item[Case (iii)]
We have $H=\mathbb{R}r\times\mathbb{R}^2$.
Since the origin is in the interior of $\mathsf{R}$, the line $\mathbb{R}r$ meets the edge of $\mathsf{R}$ opposite to $r$ in its relative interior.
So $H$ meets the corresponding edge of the fibre over $u\in U$ in a single relative-interior point.
\end{description}

 Each case contradicts the properties of $H$, which proves the theorem.
\end{proof}
\begin{corollary}\label{cor:prime-spectrum}
There is a connected-sum-prime prismatical $4$-polytope with $n$ facets if and only if $n$ is even and $n\notin\{2,4,8,10,14\}$.
\end{corollary}

\begin{proof}
By Lemma~\ref{lem:dim4-basic} the number of facets is even and at least six.
There is no example with eight facets by Proposition~\ref{prop:no-eight}, and none with ten or fourteen facets by Theorems~\ref{thm:no-prime-ten} and~\ref{thm:no-prime-fourteen}, while $\Delta^2\times\Delta^2$ is a connected-sum-prime example with six facets.
It remains to construct examples with $4k$ facets for $k\geq 3$ and with $4k+6$ facets for $k\geq 3$, since together with six these cover all even numbers outside $\{2,4,8,10,14\}$.

Let $\mathsf{S}$ be a $k$-gon with $k\geq3$ and let $\mathsf{R}$ be a triangle containing the origin in its interior.
Let $\psi$ be the piecewise linear function that equals one on the boundary of $\mathsf{S}$ and equals two at the barycentre.
Its regions of linearity are these $k$ triangles, so by Corollary~\ref{cor:prismatical} and Theorem~\ref{thm:total_space} the total space $\mathsf{P}_k$ is a prismatical $4$-polytope with $k+3k=4k$ facets, and it is connected-sum-prime by Theorem~\ref{thm:prime}.

For the second family, centrally subdivide one triangle at its barycentre $z$. 
Modify $\psi$ by setting $\psi_\varepsilon(z)=4/3+\varepsilon$ and retaining all old vertex values. 
For sufficiently small $\varepsilon>0$, the total space $\mathsf Q_k$ resulting  from this function $\psi_\varepsilon$  is a prismatical $4$-polytope with $k+3(k+2)=4k+6$ facets and is connected-sum-prime by Theorem~\ref{thm:prime}.
\end{proof}

\begin{figure}[ht]
\centering
\begin{tikzpicture}[scale=1.1, line join=round]
  \foreach \i in {0,...,5}
    \coordinate (v\i) at (60*\i:2);
  \coordinate (c) at (0,0);

  \fill[blue!10] (v0)--(v1)--(v2)--(v3)--(v4)--(v5)--cycle;
  \foreach \i in {0,...,5} \draw (c)--(v\i);
  \draw[thick] (v0)--(v1)--(v2)--(v3)--(v4)--(v5)--cycle;

  \foreach \i in {0,...,5} \fill (v\i) circle (1.5pt);
  \fill (c) circle (1.5pt) node[below=4pt, fill=blue!10, inner sep=1pt] {};
\end{tikzpicture}
\hspace{1.5cm}
\begin{tikzpicture}[scale=1.1, line join=round]
  \foreach \i in {0,...,5}
    \coordinate (v\i) at (60*\i:2);
  \coordinate (c) at (0,0);
  \coordinate (p) at (30:1.1547);

  \fill[blue!10] (v0)--(v1)--(v2)--(v3)--(v4)--(v5)--cycle;
  \fill[blue!25] (c)--(v0)--(v1)--cycle;
  \foreach \i in {0,...,5} \draw (c)--(v\i);
  \foreach \x in {c,v0,v1} \draw (p)--(\x);
  \draw[thick] (v0)--(v1)--(v2)--(v3)--(v4)--(v5)--cycle;

  \foreach \i in {0,...,5} \fill (v\i) circle (1.5pt);
  \fill (c) circle (1.5pt) node[below=4pt, fill=blue!10, inner sep=1pt] {};
  \fill (p) circle (1.5pt) node[above left=-1pt] {};
\end{tikzpicture}
\caption{The subdivisions of the hexagon used to construct $P_6$ and $Q_6$ respectively.}
\label{fig:base_subdivisions}
\end{figure}
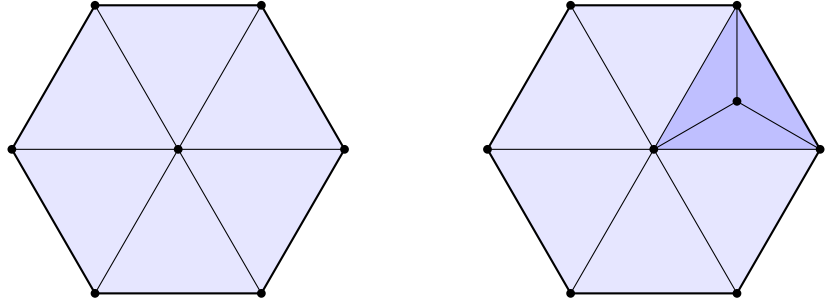

Every prismatical $4$-polytope constructed so far admits a polyhedral subdivision, without new vertices, in which every cell is combinatorially equivalent to $\mathsf{B}$.
For $\Delta^2\times\Delta^2$ this is the trivial subdivision.
If $\mathsf{P}$ and $\mathsf{P}'$ admit such subdivisions, then so does any connected sum $\mathsf{P}\cs\mathsf{P}'$.
In particular, the dual graph of the subdivision of an iterated connected sum of copies of $\mathsf{B}$ is a tree.
For a total space $\mathsf{T}$ as in Corollary~\ref{cor:prismatical}, the preimage of each triangle of the subdivision of the base in the fibration is isomorphic to $\mathsf B$.
These preimages subdivide $\mathsf{T}$, and the dual graph of this subdivision is the dual graph of the triangulation of the base.
See Figure \ref{fig:dual_graph}

\begin{figure}[ht]
\centering
\begin{tikzpicture}[scale=1.2, line join=round]
  \foreach \i in {0,...,5}
    \coordinate (v\i) at (60*\i:2);
  \foreach \j in {0,1,2}
    \coordinate (p\j) at (120*\j:0.8);

  \begin{scope}[gray!50]
    \foreach \j in {0,1,2}
      \foreach \k in {-1,0,1} {
        \pgfmathtruncatemacro{\m}{mod(2*\j+\k+6,6)}
        \draw (p\j) -- (v\m);
      }
    \draw (p0)--(p1)--(p2)--cycle;
    \draw[thick] (v0)--(v1)--(v2)--(v3)--(v4)--(v5)--cycle;
  \end{scope}

  \coordinate (c) at (0,0);                       
  \foreach \i in {0,1,2} {                        
    \pgfmathtruncatemacro{\a}{2*\i+1}
    \pgfmathtruncatemacro{\b}{mod(\i+1,3)}
    \coordinate (m\i) at ($0.3333*(v\a)+0.3333*(p\i)+0.3333*(p\b)$);
  }
  \foreach \k in {0,...,5} {                      
    \pgfmathtruncatemacro{\l}{mod(\k+1,6)}
    \pgfmathtruncatemacro{\j}{mod(ceil(\k/2),3)}
    \coordinate (b\k) at ($0.3333*(v\k)+0.3333*(v\l)+0.3333*(p\j)$);
  }

  \begin{scope}[thick]
    \foreach \i in {0,1,2} {
      \pgfmathtruncatemacro{\x}{2*\i}
      \pgfmathtruncatemacro{\y}{2*\i+1}
      \pgfmathtruncatemacro{\z}{mod(2*\i+5,6)}
      \draw (c) -- (m\i);
      \draw (m\i) -- (b\x);
      \draw (m\i) -- (b\y);
      \draw (b\z) -- (b\x);
    }
  \end{scope}

  \fill (c) circle (2pt);
  \foreach \i in {0,1,2} \fill (m\i) circle (2pt);
  \foreach \k in {0,...,5} \fill (b\k) circle (2pt);
\end{tikzpicture}
\caption{The dual graph of the triangulated hexagon of
Figure~\ref{fig:triangulated_hexagon}. It is also the dual graph of the
subdivision of the corresponding total space into ten copies of $\mathsf{B}$,
and it contains cycles.}
\label{fig:dual_graph}
\end{figure}
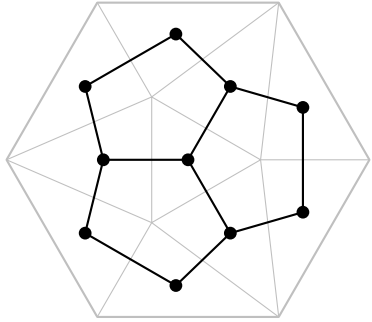

\begin{question}
Does every prismatical 4-polytope admits a polyhedral subdivision with no extra vertices where maximal cells are isomorphic to $ \mathsf{B} $?
\end{question}

The answer is negative for cubical 3-polytopes using the cube instead of $\mathsf B$.

\begin{example}\label{ex:cubical_prime}
The convex hull of the columns of the matrix
\[
\begin{bmatrix}
0 & 1 & 0 & 1 & 0 & 1 & 0 & 1           & \tfrac12 & \tfrac57  \\[2pt]
0 & 0 & 1 & 1 & 0 & 0 & 1 & \tfrac12    & 1        & \tfrac57  \\[2pt]
0 & 0 & 0 & 0 & 1 & 1 & 1 & \tfrac32    & \tfrac32 & \tfrac{12}7
\end{bmatrix}
\]
is a cubical $3$-polytope with ten vertices and eight facets; see Figure~\ref{fig:cubical-no-subdivision}.
It admits no subdivision into combinatorial cubes without new vertices.
It is not itself a cube, so such a subdivision would have two cells sharing a quadrilateral; these cells lie on opposite sides of its affine hull and therefore span $4+4+4=12$ vertices.
\end{example}

Cubical polytopes have been studied from the point of view of constructions and of subdivisions into cubes, see Schwartz~\cite{Sch04} and Schwartz and Ziegler~\cite{SchZie04}.

\begin{figure}[ht]
\centering
\begin{tikzpicture}[line join=round]
  \begin{scope}[x={(-1.05cm,-0.81cm)}, y={(1.21cm,-0.70cm)}, z={(0cm,1.19cm)}, scale=1.5]
    \coordinate (p0) at (0,0,0);
    \coordinate (p1) at (1,0,0);
    \coordinate (p2) at (0,1,0);
    \coordinate (p3) at (1,1,0);
    \coordinate (p4) at (0,0,1);
    \coordinate (p5) at (1,0,1);
    \coordinate (p6) at (0,1,1);
    \coordinate (p7) at (1,1/2,3/2);
    \coordinate (p8) at (1/2,1,3/2);
    \coordinate (p9) at (5/7,5/7,12/7);

    \fill[blue!10] (p1)--(p5)--(p7)--(p3)--cycle;
    \fill[blue!10] (p2)--(p3)--(p8)--(p6)--cycle;
    \fill[blue!10] (p3)--(p8)--(p9)--(p7)--cycle;
    \fill[blue!10] (p4)--(p9)--(p7)--(p5)--cycle;
    \fill[blue!10] (p4)--(p9)--(p8)--(p6)--cycle;

    \draw[dashed] (p0)--(p1) (p0)--(p2) (p0)--(p4);

    \draw[thick] (p1)--(p3)--(p2) (p1)--(p5) (p2)--(p6);
    \draw[thick] (p5)--(p4)--(p6);
    \draw[thick] (p3)--(p7) (p3)--(p8);
    \draw[thick] (p5)--(p7)--(p9)--(p8)--(p6) (p4)--(p9);

    \foreach \i in {0,...,9} \fill (p\i) circle (0.9pt);
  \end{scope}
  \node at (0,-2.8) {(a)};

  \begin{scope}[xshift=6.2cm, scale=0.75]
    \coordinate (s0) at ( 0,   2.8);
    \coordinate (s1) at (-2.8, 0);
    \coordinate (s2) at ( 2.8, 0);
    \coordinate (s3) at ( 0,  -2.8);
    \coordinate (s4) at ( 0,   1.4);
    \coordinate (s5) at (-1.4, 0);
    \coordinate (s6) at ( 1.4, 0);
    \coordinate (s7) at (-0.65,-0.65);
    \coordinate (s8) at ( 0.65,-0.65);
    \coordinate (s9) at ( 0,  -0.35);

    \fill[blue!10] (s0)--(s1)--(s3)--(s2)--cycle;

    \draw (s0)--(s4) (s1)--(s5) (s2)--(s6) (s3)--(s7) (s3)--(s8);
    \draw (s4)--(s5)--(s7)--(s9)--(s4) (s4)--(s6)--(s8)--(s9);

    \draw[thick] (s0)--(s1)--(s3)--(s2)--cycle;

    \foreach \i in {0,...,9} \fill (s\i) circle (1.5pt);
  \end{scope}
  \node at (6.2,-2.8) {(b)};
\end{tikzpicture}
\caption{A cubical $3$-polytope with ten vertices and eight facets.}
\label{fig:cubical-no-subdivision}
\end{figure}
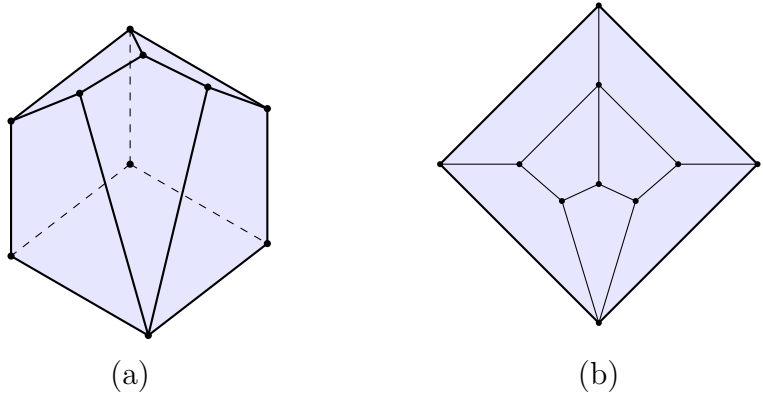

\section{Dimension greater than four}
\label{sec:dim-at-least-five}

This section is devoted to proving Theorem~\ref{thm:product-facet-1-characterization}(viii): the class $\mathscr P_d$ is empty for every $d\ge5$.
Throughout the section we argue by contradiction, so we fix $d\ge 5$ and a prismatical $d$-polytope $\mathsf P$.

Every facet of $\mathsf P$ is a simplicial prism $\Delta^{d-2}\times\Delta^1$, whose facets are of two kinds: the two \emph{simplicial ridges} $\Delta^{d-2}\times\set{0}$ and $\Delta^{d-2}\times\set{1}$, and the $d-1$ \emph{prismatical ridges}
$\Delta^{d-3}\times\Delta^1$, one for each facet of $\Delta^{d-2}$.
Consequently every ridge of $\mathsf P$ is either a $(d-2)$-simplex or a simplicial $(d-2)$-prism, and each facet of $\mathsf P$ contains exactly two simplicial ridges and exactly $d-1$ prismatical ridges.
Since two distinct facets of a polytope meet in at most one ridge, this yields two spanning subgraphs of the dual graph $G^*(\mathsf P)$:
\begin{itemize}
  \item the \emph{prismatical dual graph} $G^*_{\Pi}(\mathsf P)$, whose vertices are the facets of $\mathsf P$ and whose edges join two facets meeting in a prismatical ridge; it is $(d-1)$-regular;
  \item the \emph{simplicial dual graph} $G^*_{\Delta}(\mathsf P)$, whose vertices are the facets of $\mathsf P$ and whose edges join two facets meeting in a simplicial ridge; it is $2$-regular, hence a disjoint union of cycles.
\end{itemize}

We call the $d-1$ edges of a facet $\mathsf F\cong\Delta^{d-2}\times\Delta^1$ that join the two simplicial ridges of $\mathsf F$ the \emph{bridge edges} of $\mathsf F$; they are the edges $\set{p}\times\Delta^1$ with $p$ a vertex of $\Delta^{d-2}$.
The whole argument rests on the following elementary observation, which is where the hypothesis $d\ge5$ enters.

\begin{remark}\label{rem:intrinsic-prism-structure}
Let $k\ge2$.
The prism $\Delta^{k}\times\Delta^1$ has exactly two facets that are simplices,
namely $\Delta^{k}\times\set{0}$ and $\Delta^{k}\times\set{1}$; the remaining $k+1$ facets are prisms $\Delta^{k-1}\times\Delta^1$, which are not simplices because $k-1\ge1$.
Hence for $k\ge2$ the product decomposition of $\Delta^{k}\times\Delta^1$, and therefore its set of bridge edges, is determined by the face lattice.
For $k=1$ this fails: the square $\Delta^1\times\Delta^1$ carries two distinct prism structures.

Since $d\ge5$, the prismatical ridges of $\mathsf P$ are prisms $\Delta^{d-3}\times\Delta^1$ with $d-3\ge2$, so each prismatical ridge $\mathsf R$ has an intrinsically defined set of $d-2$ bridge edges, and these are bridge edges of both facets of $\mathsf P$ containing $\mathsf R$.
This is exactly what fails when $d=4$, where the prismatical ridges are quadrilaterals, as all the examples of Section~\ref{sec:prismatical-dim-4} illustrate.
\end{remark}

\begin{lemma}[Bridge point]\label{lem:bridge-point}
After a projective transformation we may assume that, for every facet $\mathsf F$ of $\mathsf P$, the affine hulls of the $d-1$ bridge edges of $\mathsf F$ meet in a single point $v_{\mathsf F}\in\R^d$, and that \[
  \mathsf S_{\mathsf F}:=\conv\paren*{\mathsf F\cup\set{v_{\mathsf F}}}
\]
is a $(d-1)$-simplex.
Moreover $v_{\mathsf F}\in\aff\mathsf F\setminus\mathsf F$; in particular $v_{\mathsf F}\notin\mathsf P$.
\end{lemma}

\begin{proof}
Let $\mathsf F$ be a facet of $\mathsf P$, with simplicial ridges $\mathsf T_0,\mathsf T_1$ and bridge edges $\mathsf e_1,\dots,\mathsf e_{d-1}$, where $\mathsf e_i=[a_i,b_i]$ with $a_i\in\Fc_0(\mathsf T_0)$ and $b_i\in\Fc_0(\mathsf T_1)$.
For $i\ne j$ the four vertices $a_i,a_j,b_i,b_j$ span a quadrilateral $2$-face of $\mathsf F$, so the lines $\aff\mathsf e_i$
and $\aff\mathsf e_j$ are distinct and coplanar; viewed in $\mathbb{RP}^d\supset\R^d$ they meet in exactly one point.
The bridge edges are not all coplanar, because their $2(d-1)\ge8$
endpoints affinely span $\aff\mathsf F$, of dimension $d-1\ge4$.
If two of the bridge lines meet at a point
$p$ and a third one avoids $p$, then that third line meets the plane spanned by the first two in two distinct points and hence lies in it; iterating, all the lines would be coplanar.
Therefore $\aff\mathsf e_1,\dots,\aff\mathsf e_{d-1}$ have a single common point $v_{\mathsf F}\in\mathbb{RP}^d$.

The set of projective hyperplanes disjoint from $\mathsf P$ is a nonempty open subset of the dual projective space \footnote{a set of projective hyperplanes is \emph{open} if, for every hyperplane in it, all sufficiently small changes to the coefficients of a chosen defining equation still represent hyperplanes in the set.}.
Additionally, the hyperplanes through a fixed point form a nowhere dense subset\footnote{a subset is \emph{nowhere dense} if every nonempty open set contains a nonempty open subset disjoint from it.}.
As $\mathsf P$ has finitely many facets, we may choose a projective hyperplane disjoint from $\mathsf P$ and from all the points $v_{\mathsf F}$, and apply a projective transformation sending it to the hyperplane at infinity (see \cite[Sec.~2.5]{Pin24}).
After this change of coordinates $\mathsf P$ is again a $d$-polytope in $\R^d$ and every $v_{\mathsf F}$ is an ordinary point of $\R^d$.

Fix $\mathsf F$ and write $v:=v_{\mathsf F}$.
Since $\mathsf e_i$ is a face of $\mathsf F$, we have $\mathsf F\cap\aff\mathsf e_i=\mathsf e_i$; if $v$ belonged to $\mathsf F$ it would therefore lie in $\mathsf e_1\cap\mathsf e_2=\varnothing$.
Hence $v\notin\mathsf F$, and since $v\in\aff\mathsf F$ and $\mathsf P\cap\aff\mathsf F=\mathsf F$, also $v\notin\mathsf P$.

Because $v\notin\mathsf e_i$, the points $a_i$ and $b_i$ lie on the same ray emanating from $v$; write $a_i=v+s_iu_i$ and $b_i=v+t_iu_i$ with $s_i,t_i>0$.
If $v$ lay in $\aff\mathsf T_0$
then every $b_i$ would lie in $\aff\mathsf T_0$ as well, forcing $\mathsf F\subseteq\aff\mathsf T_0$, which is absurd.
Hence $\set{v,a_1,\dots,a_{d-1}}$ is affinely independent and $\mathsf S:=\conv\set{v,a_1,\dots,a_{d-1}}$ is a $(d-1)$-simplex.
For $i\ne j$ the quadrilateral $2$-face $[a_i,a_j,b_j,b_i]$ has $[a_i,a_j]$ and $[b_i,b_j]$ among its edges, so these two segments do not cross; consequently $s_i-t_i$ and $s_j-t_j$ have the same sign.
Interchanging the labels of $\mathsf T_0$ and $\mathsf T_1$ if necessary, we may assume $t_i<s_i$ for all $i$, so that $b_i\in\ri[v,a_i]$ and $\mathsf F\subseteq\mathsf S$.
Therefore $\conv(\mathsf F\cup\set{v})=\mathsf S$ is a $(d-1)$-simplex, and $\mathsf F$ is obtained from $\mathsf S$ by truncating the vertex $v$, with $\mathsf T_0$ the facet of $\mathsf S$ opposite $v$ and $\ri\mathsf T_1\subseteq\ri\mathsf S$.
\end{proof}

From now on we assume that the coordinates have been normalised as in Lemma~\ref{lem:bridge-point}.
For a facet $\mathsf F$ we keep the notation $\mathsf T_0(\mathsf F)$ and $\mathsf T_1(\mathsf F)$ for the simplicial ridges of $\mathsf F$,
labelled so that $\mathsf T_0(\mathsf F)$ is the facet of $\mathsf S_{\mathsf F}$ opposite $v_{\mathsf F}$ and $\ri\mathsf T_1(\mathsf F)\subseteq\ri\mathsf S_{\mathsf F}$.

\begin{lemma}\label{lem:bridge-point-constant}
Let $\mathsf F$ and $\mathsf G$ be facets of $\mathsf P$ lying in the same connected component of $G^*_{\Pi}(\mathsf P)$.
Then $v_{\mathsf F}=v_{\mathsf G}$.
\end{lemma}

\begin{proof}
Assume first that $\mathsf F$ and $\mathsf G$ are adjacent in $G^*_{\Pi}(\mathsf P)$, and let $\mathsf R:=\mathsf F\cap\mathsf G$ be the prismatical ridge they share.
By Remark~\ref{rem:intrinsic-prism-structure} the $d-2$ bridge edges of $\mathsf R$ are bridge edges of $\mathsf F$ and also bridge edges of $\mathsf G$.
Since $d-2\ge3$, at least two of them are available, and the common point of the affine hulls of any two distinct bridge edges is unique; hence $v_{\mathsf F}=v_{\mathsf G}$.
The general case follows by induction on the distance from $\mathsf F$ to $\mathsf G$ in $G^*_{\Pi}(\mathsf P)$.
\end{proof}

For a facet $\mathsf F$ of $\mathsf P$ we write $U_{\mathsf F}$ for the set of facets lying in the connected component of $\mathsf F$ in $G^*_{\Pi}(\mathsf P)$, and $v:=v_{\mathsf F}=v_{\mathsf G}$ for every $\mathsf G\in U_{\mathsf F}$.

\begin{lemma}\label{lem:cone-over-component}
Let $\mathsf F$ be a facet of $\mathsf P$, let $v:=v_{\mathsf F}$ and let $\mathsf Q:=\conv(\mathsf P\cup\set{v})$.
Then \begin{enumerate}[{\rm (i)}]
  \item $v$ is a vertex of $\mathsf Q$;
  \item the facets of $\mathsf Q$ containing $v$ are exactly the $(d-1)$-simplices $\mathsf S_{\mathsf G}=\conv(\mathsf G\cup\set{v})$ with $\mathsf G\in U_{\mathsf F}$,
        and $\mathsf G\mapsto\mathsf S_{\mathsf G}$ is a bijection from $U_{\mathsf F}$ onto this set of facets.
\end{enumerate}
\end{lemma}

\begin{proof}
(i) By Lemma~\ref{lem:bridge-point}, $v\notin\mathsf P$, so $v$ is a vertex of $\conv(\mathsf P\cup\set{v})$.

(ii) Let $\mathsf G\in U_{\mathsf F}$.
By Lemma~\ref{lem:bridge-point-constant} we have $v\in\aff\mathsf G$, so the supporting hyperplane $\aff\mathsf G$ of $\mathsf P$ also supports $\mathsf Q$, and \[
  \mathsf Q\cap\aff\mathsf G =\conv\paren*{(\mathsf P\cap\aff\mathsf G)\cup\set{v}}
  =\conv(\mathsf G\cup\set{v})
  =\mathsf S_{\mathsf G},
\]
which by Lemma~\ref{lem:bridge-point} is a $(d-1)$-simplex and hence a facet of $\mathsf Q$
containing $v$.
Distinct facets of $\mathsf P$ have distinct affine hulls, so $\mathsf G\mapsto\mathsf S_{\mathsf G}$ is injective.

It remains to show that there are no further facets of $\mathsf Q$ through $v$.
Consider the vertex figure $\mathsf Q/v$, a $(d-1)$-polytope whose facets correspond to the facets of $\mathsf Q$ containing $v$ and whose ridges correspond to the ridges of $\mathsf Q$ containing $v$; its dual graph is connected.
Fix $\mathsf G\in U_{\mathsf F}$.
The facets of the simplex $\mathsf S_{\mathsf G}$ that contain $v$ are the $d-1$ sets $\conv(\mathsf R\cup\set{v})$ with $\mathsf R$ a prismatical ridge of $\mathsf G$; each of them is contained in $\mathsf S_{\mathsf G'}$, where $\mathsf G'$ is the other facet of $\mathsf P$
containing $\mathsf R$, and $\mathsf G'\in U_{\mathsf F}$.
Hence, in the dual graph of $\mathsf Q/v$, all $d-1$ neighbours of the vertex corresponding to $\mathsf S_{\mathsf G}$
already correspond to facets $\mathsf S_{\mathsf G'}$ with $\mathsf G'\in U_{\mathsf F}$.
Thus $\set{\mathsf S_{\mathsf G}:\mathsf G\in U_{\mathsf F}}$ is a union of connected components of the dual graph of $\mathsf Q/v$; as that graph is connected, it is the whole vertex set.
\end{proof}

\begin{corollary}[Separation]\label{cor:separation}
Let $\mathsf F$ be a facet of $\mathsf P$, let $v:=v_{\mathsf F}$, and let $\mathsf Q:=\conv(\mathsf P\cup\set{v})$.
Then the set of facets of $\mathsf P$ not in $U_{\mathsf F}$ is the disjoint union of the set $A_{\mathsf F}$ of facets visible from $v$ and the set $B_{\mathsf F}$ of facets nonvisible from $v$.
Both sets are nonempty and no ridge of $\mathsf P$ is contained in a facet of $A_{\mathsf F}$ and in a facet of $B_{\mathsf F}$.
In particular, deleting $U_{\mathsf F}$ disconnects $G^*(\mathsf P)$.
\end{corollary}

\begin{proof}

Let $\mathsf G$ be a facet of $\mathsf P$ with $\mathsf G\notin U_{\mathsf F}$.
If $v\in\aff\mathsf G$, then, as in the proof of Lemma~\ref{lem:cone-over-component}(ii),
$\conv(\mathsf G\cup\set{v})$ would be a facet of $\mathsf Q$ containing $v$, hence equal to some $\mathsf S_{\mathsf G'}$ with $\mathsf G'\in U_{\mathsf F}$; comparing affine hulls gives $\mathsf G=\mathsf G'$, a contradiction.
So $v\notin\aff\mathsf G$ and $\mathsf G$ is either visible or nonvisible from $v$.
Note also that if $\mathsf G$ is nonvisible from $v$, then $\aff\mathsf G$ supports $\mathsf Q$ and $\mathsf Q\cap\aff\mathsf G=\mathsf G$, so $\mathsf G$
is a facet of $\mathsf Q$.

Since $v\notin\mathsf P$, at least one facet of $\mathsf P$ is visible from $v$, so $A_{\mathsf F}\ne\varnothing$.
For $B_{\mathsf F}$, fix $\mathsf G\in U_{\mathsf F}$ and put $\mathsf T:=\mathsf T_0(\mathsf G)$, a ridge of $\mathsf Q$ with $v\notin\mathsf T$.
Let $\mathsf W$ be the facet of $\mathsf Q$ containing $\mathsf T$ other than $\mathsf S_{\mathsf G}$.
If $v\in\mathsf W$ then $\mathsf W=\mathsf S_{\mathsf G''}$ for some $\mathsf G''\in U_{\mathsf F}$, whence $\mathsf S_{\mathsf G}=\conv(\mathsf T\cup\set{v})\subseteq\mathsf W$ and $\mathsf S_{\mathsf G}=\mathsf W$, a contradiction.
So $v\notin\mathsf W$; writing $\mathsf W=\mathsf Q\cap H$ for a supporting hyperplane $H$ of $\mathsf Q$ with $v\notin H$, we get $\mathsf W=\mathsf P\cap H$, a facet of $\mathsf P$ nonvisible from $v$.
As $v\notin\aff\mathsf W$, we have $\mathsf W\notin U_{\mathsf F}$, so $B_{\mathsf F}\ne\varnothing$.

Suppose now that a ridge $\mathsf R$ of $\mathsf P$ is contained in $\mathsf G_A\in A_{\mathsf F}$ and in $\mathsf G_B\in B_{\mathsf F}$.
Let $h_A(x)\le \beta_A$ and $h_B(x)\le\beta_B$ be the facet inequalities of $\mathsf G_A$ and $\mathsf G_B$, so that $h_A(v)>\beta_A$ and $h_B(v)<\beta_B$.
For $\lambda\in(0,1)$ the inequality $\lambda h_A+(1-\lambda)h_B\le\lambda\beta_A+(1-\lambda)\beta_B$ is valid on $\mathsf P$ with equality set $\mathsf G_A\cap\mathsf G_B=\mathsf R$.
By the intermediate value theorem there is $\lambda^*\in(0,1)$ for which $v$ satisfies this inequality with equality; the corresponding hyperplane supports $\mathsf Q$ and meets it in $\conv(\mathsf R\cup\set{v})$, which has dimension $d-1$ because $v\notin\aff\mathsf G_A \supseteq\aff\mathsf R$.
Hence $\conv(\mathsf R\cup\set{v})$ is a facet of $\mathsf Q$
containing $v$, so by Lemma~\ref{lem:cone-over-component} it equals $\mathsf S_{\mathsf G}$ for some $\mathsf G\in U_{\mathsf F}$.
Then $\mathsf R\subseteq\mathsf P\cap\aff\mathsf S_{\mathsf G}=\mathsf G$, so $\mathsf R$ is a ridge of $\mathsf G$ and $\mathsf G\in\set{\mathsf G_A,\mathsf G_B}$, contradicting $\mathsf G\in U_{\mathsf F}$.

\end{proof}

\begin{theorem}\label{thm:no-prismatical-high-dim}
For every $d\ge5$ the class $\mathscr P_d$ is empty; that is, there is no $d$-polytope all of whose facets are simplicial prisms.
\end{theorem}

\begin{proof}
 Suppose that $\mathsf P$ is a prismatical $d$-polytope with $d\ge5$, normalised as in Lemma~\ref{lem:bridge-point}.
Let $U_1,\dots,U_k$ be the vertex sets of the connected components of $G^*_{\Pi}(\mathsf P)$.
Define a graph $\Gamma$ with vertex set $\set{1,\dots,k}$, where distinct vertices $i,j$ are adjacent if some facet in $U_i$ shares a ridge with some facet in $U_j$.
Since $G^*(\mathsf P)$ is connected, $\Gamma$ is connected: replacing each facet along a path in $G^*(\mathsf P)$ from a facet in $U_i$ to a facet in $U_j$ by the index of its component, and suppressing consecutive repetitions, gives a walk from $i$ to $j$ in $\Gamma$.

Fix $j$.
By Corollary~\ref{cor:separation}, the facets outside $U_j$ are partitioned into two nonempty sets $A_j$ and $B_j$, with no edge of $G^*(\mathsf P)$ joining these sets.
Every component $U_i$ with $i\ne j$ lies entirely in one of them; otherwise a path in $G^*_{\Pi}(\mathsf P)$ within $U_i$ would contain an edge joining $A_j$ to $B_j$, a contradiction.

Consequently, the vertices of $\Gamma$ other than $j$ are partitioned into the two nonempty sets \[
 \set{i\ne j:U_i\subseteq A_j},
 \qquad \set{i\ne j:U_i\subseteq B_j}.
\]
No edge of $\Gamma-j$ joins these sets, so $\Gamma-j$ is disconnected.
Since $j$ was arbitrary, deleting any vertex of $\Gamma$ disconnects it.

This is impossible for a finite connected graph.
To see this, choose a spanning tree of $\Gamma$ and a leaf of that tree.
Deleting the leaf leaves a connected tree spanning all the remaining vertices, so deleting the same vertex from $\Gamma$ cannot disconnect it.
This contradiction proves the theorem.
\end{proof}

Theorem~\ref{thm:no-prismatical-high-dim} settles part~(viii) of Theorem~\ref{thm:product-facet-1-characterization} and completes its proof.


\section{Acknowledgements}
\label{sec:acknowledgments} 
The authors thank Raman Sanyal and Florian Rieg for helpful comments.
Florian Rieg brought our attention to the cubical polytope of Example \ref{ex:cubical_prime}.
This research began during Guillermo Pineda Villavicencio’s visit to Federico Castillo and Jos\'e Samper at the Pontificia Universidad Católica de Chile in January 2026.
He thanks them for their hospitality. 

\subsection{Declaration on the use of generative AI and AI-assisted technologies}
\label{subsec:acknowledments-ai} 

During the preparation of this work, the authors used OpenAI's GPT-5.6 Sol \cite{OpenAIChatGPT} and Anthropic's Claude Opus 4.8 \cite{AnthropicClaude} models to assist with proofreading and refining the presentation of parts of the proofs in Section~\ref{sec:prismatical-dim-4} and the Appendix.
GPT-5.6 Sol was also used to generate SageMath \cite{SAGE} code supporting the construction and verification of the  families of prismatical $4$-polytopes in Section~\ref{subsec:polytope-families}.

\appendix
\section{Prismatical 4-polytopes with 14 facets}

The purpose of this section is to prove the following theorem.

\begin{theorem}\label{thm:no-prime-fourteen}
Every prismatical 4-polytope with fourteen facets is combinatorially isomorphic to a connected sum of three copies of $\Delta^2\times\Delta^2$. 
In particular, there is no connected-sum-prime example. 
\end{theorem}

For the proof of Theorem~\ref{thm:no-prime-fourteen}, we first establish three auxiliary lemmas.  

\begin{lemma}\label{lem:match} Let $\mathsf P$ be a prismatical 4-polytope, and put $n_i:=\#\{u\in\Fc_0(\mathsf P):f_3(u,\mathsf P)=i\}$.
Suppose that every vertex of $\mathsf P$ is incident with either four or six facets and that no triangular ridge has all three vertices simple.
Then $n_6\ge2n_4$ and $n_6-2n_4\equiv0\pmod4$.
\end{lemma}

\begin{proof}
For an edge $\mathsf e$ of $\mathsf P$,  let $\tau(\mathsf e)$ be the number of triangular ridges containing $\mathsf e$.
For a vertex $u$ of $\mathsf P$, let $\mathsf F_u$ be the facet of the dual polytope $\mathsf P^*$ conjugate to $u$.
Then $\mathsf F_u$ is isomorphic to $(\mathsf P/u)^*$, so $\mathsf F_u$  is simple.
Moreover, $\mathsf F_u$ has $f_3(u,\mathsf P)$ vertices: it is a tetrahedron when $f_3(u,\mathsf P)=4$ and a triangular prism when $f_3(u,\mathsf P)=6$.

Call an edge of $\mathsf F_u$ \emph{key} if its conjugate edge in $\mathsf P/u$ corresponds to a triangular ridge of $\mathsf P$ containing $u$.
Let $v$ be a vertex of $\mathsf F_u$,
and let $\mathsf I_v$ be the corresponding prism facet of $\mathsf P$.
The edges of $\mathsf F_u$ incident with $v$ correspond to the ridges of $\mathsf I_v$ containing $u$.
Since every vertex of a 
prism lies in exactly one of its two triangular faces, precisely one of these ridges is triangular.
Hence precisely one key edge is incident with $v$.
Thus the key edges form a perfect matching in $\mathsf F_u$.

For an edge $[u,x]$ of $\mathsf P$, let $\mathsf G_{x}$ be the face of $\mathsf F_u$ conjugate to the corresponding vertex of $\mathsf P/u$.
Then \begin{equation}\label{eq:key}
 \tau([u,x])=\#\{\text{key edges contained in }\mathsf G_x\}.
\end{equation}
If $u$ is simple, $\mathsf F_u$ is a tetrahedron.
Its key edges form a perfect matching,
and each triangular face contains exactly one of them.
Thus all four edges incident with $u$ are simple and have $\tau=1$.
If $u$ is incident with 6 facets in $\mathsf P$, $\mathsf F_u$ is a prism.
Up to automorphism, $G(\mathsf F_u)$  has two perfect matchings: \emph{type $A$}, consisting of the three vertical edges, and \emph{type $B$},
consisting of one vertical edge and one edge in each triangular face; see Figure~\ref{fig:matchings}.
We call $u$ \emph{type $A$} or \emph{type $B$} accordingly.
Formula \eqref{eq:key} gives \begin{table}[htbp]
    \centering $\begin{array}{l|l|l}
 &\text{simple edges}&\text{nonsimple edges}\\ \hline \text{simple vertex}&4\text{ with }\tau=1&\text{none}\\
 \text{type }A&2\text{ with }\tau=0&3\text{ with }\tau=2\\
 \text{type }B&2\text{ with }\tau=1&
   2\text{ with }\tau=1,\ 1\text{ with }\tau=2.
\end{array}$
    \caption{Classification of vertices in $\mathsf P$ according to the matching of key edges in $\mathsf F_u$.}
    \label{tab:key-edge-classification}
\end{table}

For instance, for a type $A$ matching, neither triangular face of $\mathsf F_u$ contains a key edge, while every quadrilateral face contains two.  
\begin{figure}[htbp]             
\begin{center}     
\includegraphics[scale=1]{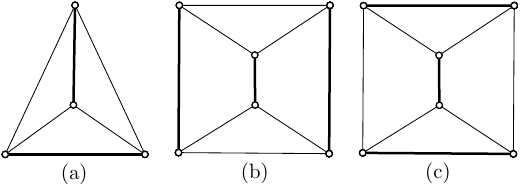}  
\end{center}
\caption{Matchings in $\mathsf F_u$ highlighted in bold. {(a)} A matching in a tetrahedron. (b) A type $A$ matching in the prism $\mathsf F_u$. (c) A type $B$ matching in the prism $\mathsf F_u$.}
\label{fig:matchings}   
\end{figure} 

Let $\mathsf K$ be the spanning subgraph of $G(\mathsf P)$ whose edges are the nonsimple edges with $\tau=2$.
If $u$ is a simple vertex in $\mathsf P$, then $\deg_{\mathsf K}(u)=0$ (see \eqref{eq:key} and Table~\ref{tab:key-edge-classification}).
If  $u$ has type $A$ then $\deg_{\mathsf K}(u)=3$, while  if  $u$ has type $B$ then $\deg_{\mathsf K}(u)=1$; see \eqref{eq:key} and Table~\ref{tab:key-edge-classification}. 

\begin{figure}   
\begin{center}
\includegraphics[scale=1]{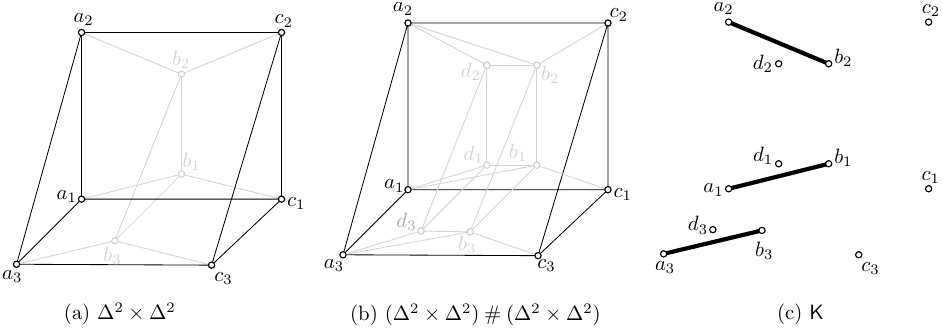}   
\end{center}
\caption{The prismatical 4-polytope $\Delta^2\times \Delta^2$ and connected sums. {(a)} A Schlegel diagram of $\Delta^2\times\Delta^2$.  (b) A Schlegel diagram of the prismatical $4$-polytope $\mathsf P=(\Delta^2\times\Delta^2)\mathbin{\#}(\Delta^2\times\Delta^2)$.
The images in $\mathsf P$ of the two summands have vertex sets   $\{a_1,a_2,a_3,b_1,b_2,b_3,c_1,c_2,c_3\}$ and $\{a_1,a_2,a_3,b_1,b_2,b_3,d_1,d_2,d_3\}$, respectively; the common labels $a_i,b_i$ denote the identifications made along the two prism facets used in forming the connected sum. (c) The spanning subgraph $\mathsf K$ of $G(\mathsf P)$, with its edges highlighted in bold.
Lemma~\ref{lem:match} does not apply to $\mathsf P$, because the triangular ridges $[c_1,c_2,c_3]$ and $[d_1,d_2,d_3]$ have all three vertices simple.}
\label{fig:B-BcsB-subgraphK}   
\end{figure}

We now classify the nontrivial connected components of $\mathsf K$.
Suppose that $u$ has type $A$, and let $v_1,v_2,v_3$ be its three $\mathsf K$-neighbours.
These vertices correspond to the three quadrilateral faces  $\mathsf G_1, \mathsf G_2, \mathsf G_3$ of $\mathsf F_u$.
Each $v_i$ is nonsimple, because $[u,v_i]$ is a nonsimple edge; hence each $v_i$ has type $A$ or type $B$.
For $i\ne j$,  the faces $\mathsf G_i$ and $\mathsf G_j$ intersect in a vertical key edge.
Consequently,  $[u,v_i,v_j]$ is a triangular ridge of $\mathsf P$.
Suppose that one of the $v_i$,  say $v_1$, has type $B$; see Figure~\ref{fig:matchings}(c).
The edge $[u,v_1]$ belongs to $\mathsf K$, so the corresponding face $\mathsf G_u$ of $\mathsf F_{v_1}$ is the unique quadrilateral face containing two key edges.
Each of these key edges is contained, besides $\mathsf G_u$, in one of the two triangular faces of $\mathsf F_{v_1}$.
Hence $[v_1,v_2]$ and $[v_1,v_3]$ are simple.
Since each of these edges lies in a triangular ridge, neither $v_2$ nor $v_3$ can have type $A$.
They therefore have type $B$, and the component containing $u$ is $K_{1,3}$.
If none of $v_1,v_2,v_3$ has type $B$, then all three have type $A$.
For $i\ne j$, the edge $[v_i,v_j]$ lies in the triangular ridge $[u,v_i,v_j]$ and therefore has positive $\tau$-value.
It cannot be a simple edge at a type $A$ vertex, so it is nonsimple and has $\tau=2$.
Thus all three edges $[v_i,v_j]$ belong to $\mathsf K$, and the component containing $u$ is $K_4$.
Finally, every nontrivial component containing no type $A$ vertex is a single edge joining two type $B$ vertices, and hence is a $K_2$-component.
The remaining components of $\mathsf K$ are the isolated simple vertices. 

It remains to count the $K_2$-components in $\mathsf K$.
Let $\mathsf T=[u,v,w]$ be a triangular ridge through a simple vertex $u$.
The edges $[u,v]$ and $[u,w]$ are simple.
By hypothesis, at least one of $v,w$ is nonsimple; say $v$.
Since $[u,v]$ lies in $\mathsf T$, we have $\tau([u,v])=1$.
Hence $v$ cannot have type $A$ and must have type $B$. 

In $\mathsf F_{v}$, the triangular face $\mathsf G_u$  associated with $[v,u]$ intersects the face $\mathsf G_w$ associated with $[v,w]$  in a key edge.
In a type $B$ matching,  the key edge contained in a triangular face is also contained in a quadrilateral face containing two key edges;  see Figure~\ref{fig:matchings}.
Consequently, $\mathsf G_w$ is a quadrilateral and contains two key edges.
Thus  $[v,w]$ is a nonsimple edge with $\tau([v,w])=2$ and is the unique $\mathsf K$-edge incident with $v$.
In particular, $w$ is nonsimple.
Moreover,
$[u,w]$ is simple and has $\tau([u,w])=1$.
Hence $w$ cannot have type $A$ and must have type $B$.
Since a type $B$ vertex has degree one in $\mathsf K$, $[v,w]$  is  a $K_2$-component of $\mathsf K$. 

Conversely, let $[v,w]$ be the edge of a $K_2$-component.
Both $v$ and $w$ have type $B$ and $\tau([v,w])=2$.
Consider either of the two triangular ridges $\mathsf T=[u,v,w]$ containing $[v,w]$.
The  type $B$
matching at $v$ and $w$ shows that $[u,v]$ and $[u,w]$ are simple.
The vertex $u$ cannot have type $A$, because every simple edge at a type $A$ vertex has $\tau=0$,  whereas both $[u,v]$ and $[u,w]$
 have $\tau =1$.
Nor can $u$ have type $B$: in that case the faces $\mathsf G_v$ and $\mathsf G_w$ of $\mathsf F_u$ would be its two triangular faces, which are disjoint, whereas the ridge $\mathsf T$
would require their intersection to be a key edge.
Thus $u$ is simple. 

Consequently, sending $(u,\mathsf T)$ to $(\mathsf e,\mathsf T)$, where $\mathsf e$ is the edge of $\mathsf T$ opposite $u$, gives a bijection between the incidences of simple vertices with triangular ridges and the incidences of $K_2$-components with triangular ridges.
Each simple vertex belongs to exactly two triangular ridges, and the edge of each $K_2$-component belongs to exactly  two triangular ridges.
Hence \[
 2n_4=2\,\#\{K_2\text{-components}\}.
\]
Thus the number of $K_2$-components is $n_4$.
Let $a$ and $b$ be the numbers of $K_4$-components and $K_{1,3}$-components, respectively.
The vertices in $\mathsf P$ incident with six facets are precisely the vertices in the nontrivial components of $\mathsf K$.
Therefore \[
 n_6=4a+4b+2n_4.
\]
Hence $n_6\ge2n_4$ and $n_6-2n_4\equiv0\pmod4$.
\end{proof}

\begin{remark}
Figure~\ref{fig:B-BcsB-subgraphK}(c) shows that the hypothesis of excluding triangular ridges having three simple vertices in Lemma~\ref{lem:match} is essential.
In that case there are two such ridges: $[c_1,c_2,c_3]$ and $[d_1,d_2,d_3]$.
\end{remark}

Let $0\leq s\leq d-1$.
A $d$-polytope is \emph{$s$-simplicial} if every $s$-face is a simplex.

\begin{lemma}\label{lem:adjacent-simple-vertices}
Let $d\ge3$.
A 2-simplicial $d$-polytope with two adjacent simple vertices is a simplex.
\end{lemma}

\begin{proof}
Let $u,v$ be adjacent simple vertices of a 2-simplicial $d$-polytope $\mathsf P$, and let $S$
consist of the $d-1$ neighbours of $u$ other than $v$.
Since $u$ is simple and every $2$-face is triangular, each vertex of $S$ is adjacent to $v$.
As $v$ is also simple, neither $u$
nor $v$ has a neighbour outside $S\cup\set{u,v}$.
If $\mathsf P$ were not a simplex, it would have a vertex outside $S\cup\set{u,v}$,
so $S$ would separate $G(\mathsf P)$.
This contradicts Balinski's theorem \cite[Thm.~4.1.1]{Pin24}.
\end{proof}

Throughout the remainder of this subsection, simplicity of a vertex or edge refers to the ambient $4$-polytope.

\begin{lemma}[Simple vertices of the ambient polytope contained in a facet]
\label{lem:fourteen-simple-facet}
Let $\mathsf P$ be a prismatical $4$-polytope with no triangular ridge whose three vertices are simple, and let $\mathsf F$ be a facet of $\mathsf P$.
The following hold.
\begin{enumerate}[{\rm (i)}]
    \item  Every nonsimple vertex of $\mathsf F$ is incident with at least two nonsimple edges of $\mathsf F$.
\item The facet $\mathsf F$ has at most two simple vertices of $\mathsf P$,
and at most one if some vertex of $\mathsf F$ is incident with three nonsimple edges of $\mathsf F$.
\end{enumerate}
\end{lemma}
\begin{proof} (i)
Let $u$ be a nonsimple vertex of $\mathsf F$.
The vertex figure $\mathsf P/u$ is simplicial (Lemma~\ref{lem:prismatical-dual-basics}) and is not a tetrahedron.
The three edges of $\mathsf F$ at $u$ correspond to the three vertices of a triangular facet of $\mathsf P/u$.
A simple edge incident with $u$ corresponds to a vertex of degree three in $\mathsf  P/u$.
Lemma~\ref{lem:adjacent-simple-vertices} therefore shows that at least two of these three edges are nonsimple. 

(ii) By the hypothesis of the lemma, $\mathsf F$ has at least one nonsimple vertex, in which case Part (i) ensures that it has at least three.
If $\mathsf F$ has exactly three nonsimple vertices, they are pairwise adjacent, leaving the opposite triangle in $\mathsf F$ with three simple vertices, contrary to the hypothesis.
Hence $\mathsf F$ has at least four nonsimple vertices, and so at most two simple ones.

Finally, deleting two vertices from a triangular-prism graph leaves at most $9-3-3+1=4$ edges.
If $\mathsf F$ has two simple vertices, delete them.
By Part (i), each of the four remaining vertices is still incident with at least two nonsimple edges, and hence with exactly two.
Thus, if some vertex of $\mathsf F$ is incident with three nonsimple edges of $\mathsf F$, then $\mathsf F$ has at most one simple vertex.
\end{proof}

We are finally ready to prove the nonexistence of connected-sum-prime prismatical 4-polytopes with fourteen facets.

\begin{proof}[Proof of Theorem~\ref{thm:no-prime-fourteen}]
Let $\mathsf P$ have fourteen facets.
By Lemma~\ref{lem:split} and Theorem~\ref{thm:no-prime-ten}, it suffices to find a triangular ridge whose three vertices are simple.
Suppose that no such ridge exists.
Put $n_i:=\#\set{u\in\Fc_0(\mathsf P):f_3(u,\mathsf P)=i}$. 
By \eqref{eq:counts}, $\sum_{u\in \Fc_0(\mathsf P)}  f_3(u,\mathsf P)=84$.

\textbf{Restricting the vertex figures.} Since each simple vertex belongs to four facets, and  each facet contains at most two simple vertices (Lemma~\ref{lem:fourteen-simple-facet}), we find that $4n_4\leq28$, so $n_4\leq7$.

By Lemma~\ref{lem:dim4-basic},   $f_3(u,\mathsf P)$ is even for each vertex $u$, and so is at least six at a nonsimple vertex.
Consequently, 
\[
 \sum_{u\in \Fc_0(\mathsf P)}  f_3(u,\mathsf P)=84\geq4n_4+6(f_0-n_4)=6f_0-2n_4\geq6f_0-14.
\]
Thus $f_0\leq16$, and Lemma~\ref{lem:dim4-basic} gives $f_3(u,\mathsf P)\in\set{4,6,8}$.
Subtracting $4f_0=4n_4+4n_6+4n_8$ from $4n_4+6n_6+8n_8=84$ first and then using $n_4=f_0-n_6-n_8$ yield 
\begin{equation}
\label{eq:fourteen-counts}
 n_6=42-2f_0-2n_8,\qquad n_4=3f_0-42+n_8.
\end{equation}
If $n_8=0$, these equations imply \[
 n_6-2n_4=126-8f_0\equiv2\pmod4,
\]
contrary to Lemma~\ref{lem:match}.
Hence $n_8>0$.

We call a vertex $u$ with $f_3(u,\mathsf P)=8$ \emph{octahedral} or \emph{stacked} according to its vertex figure in \eqref{eq:figs}.
For a facet $\mathsf F$ containing $u$, the nonsimple edges of $\mathsf F$ at $u$ correspond to the vertices of degree greater than three in the associated triangular facet of $\mathsf P/u$.
Reading the triangular facets of the four possible vertex figures gives Table~\ref{tab:nonsimple-edges}.

\begin{table}[htbp]   
\centering \begin{tabular}{c|c}
 \text{Vertex figure at }$u$
 & \text{number of nonsimple edges per facet at }$u$\\ \hline \text{tetrahedron}&0\text{ in all four facets}\\
 \text{triangular bipyramid}&2\text{ in all six facets}\\
 \text{octahedron}&3\text{ in all eight facets}\\
 \text{stacked type}&2\text{ in six facets, }3\text{ in two facets}.
\end{tabular}
   \caption{Nonsimple edges at a vertex $u$ of $\mathsf P$ depending on the vertex figure $\mathsf P/u$.}
    \label{tab:nonsimple-edges}
\end{table}

Figure~\ref{fig:fourteen-nonsimple-edges} depicts the rows in Table~\ref{tab:nonsimple-edges} related to a nonsimple vertex. 

For a stacked vertex $u$, call an incidence $(u,\mathsf F)$ \emph{special} if the three edges of $\mathsf F$ at $u$ are nonsimple.
There are exactly two such incidences.
In Figure~\ref{fig:fourteen-nonsimple-edges}(c), the facets of $\mathsf P$ involved in a special incidence  correspond to  the triangular faces $[v_1,w_1,w_2]$ and $[v_2,w_1,w_2]$ of $\mathsf{P}/u$. 

\begin{figure}[htbp]             
\begin{center}     
\includegraphics[scale=1.1]{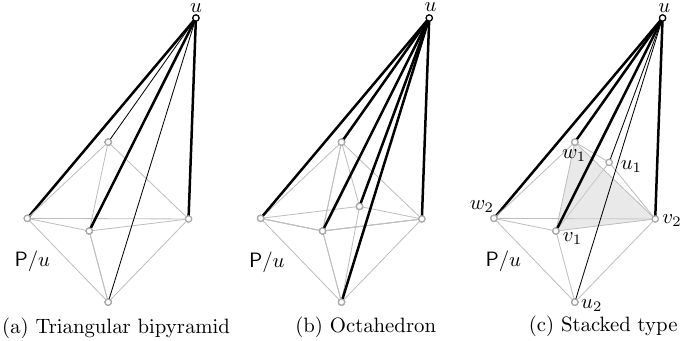}  
\end{center}
\caption{Nonsimple edges at a vertex $u$ of $\mathsf P$ depending on the vertex figure $\mathsf P/u$.
Nonsimple edges are highlighted in bold, simple edges are thin, and edges of a vertex figure are coloured in grey. (a) $\mathsf P/u$ is a triangular bipyramid. (b) $\mathsf P/u$ is an octahedron. (c) $\mathsf P/u$ is of stacked type.
The shaded triangle is the facet of the triangular bipyramid on which the additional vertex is stacked.
In (c), the only triangular facets containing no degree-three vertex are $[v_1,w_1,w_2]$ and $[v_2,w_1,w_2]$, which give the two special incidences in Table~\ref{tab:nonsimple-edges}.}
\label{fig:fourteen-nonsimple-edges}   
\end{figure}

Since $n_8>0$, some facet has a vertex incident with three nonsimple edges of that facet.
By Lemma~\ref{lem:fourteen-simple-facet}, this facet has at most one simple vertex.
Thus $4n_4\leq27$, and hence $n_4\leq6$.
On the other hand, \eqref{eq:loc} at a vertex with $f_3(u,\mathsf P)=8$ gives $f_0\geq15$.
Together with \eqref{eq:fourteen-counts}, we get $n_4=3f_0-42+n_8\le 6$, which forces \begin{equation}
\label{eq:fourteen-reduction}
 f_0=15,\qquad (n_4,n_6,n_8)=(3+k,12-2k,k),\qquad 1\leq k\leq3.
\end{equation}

\textbf{Parity in a facet.} In each facet $\mathsf F$ of $\mathsf{P}$, we count incidences between nonsimple edges and nonsimple vertices.
Each nonsimple edge has precisely two nonsimple endvertices, while each nonsimple vertex is incident with two or three nonsimple edges of $\mathsf F$.
The total number of incidences is therefore even, so an even number of vertices of $\mathsf F$ are incident with three nonsimple edges.
By the preceding table, these vertices are precisely the octahedral vertices of $\mathsf F$ and the stacked vertices whose incidence with $\mathsf F$ is special.
Hence \begin{equation}
\label{eq:fourteen-parity}
 \#\set{\text{octahedral vertices in }\mathsf F}
 +\#\set{\text{special incidences with }\mathsf F}
 \equiv0\pmod2.
\end{equation}
An edge $\mathsf e$ with $f_3(\mathsf e,\mathsf P)=5$ can have only stacked endvertices, and every stacked vertex is incident with exactly two such edges, by \eqref{eq:figs}.
These edges thus form a $2$-regular graph on the stacked vertices.
Since there are at most three vertices with $f_3(u,\mathsf P)=8$ (see \eqref{eq:fourteen-reduction}), either there are no stacked vertices or there are exactly three, forming a triangle. 

\textbf{No stacked vertices.} In this subcase, all $k$ vertices with $f_3(u,\mathsf P)=8$ are octahedral, so every facet contains an even number of them by \eqref{eq:fourteen-parity}.
If $k=1$, this is impossible.
If $k=2$, every facet contains either both or neither.
The two vertices would then belong to exactly the same facets, which is impossible because a vertex is the intersection of all facets containing it.

If $k=3$, by \eqref{eq:fourteen-parity} each facet containing octahedral vertices contains precisely two such vertices.
Their $3\cdot8=24$ incidences therefore occur in twelve facets.
By Table~\ref{tab:nonsimple-edges}, in each of these facets, an octahedral vertex is incident with three nonsimple edges.
Thus, by Lemma~\ref{lem:fourteen-simple-facet}, each of the twelve facets with octahedral vertices has at most one simple vertex, and each of the two remaining facets  has at most two.
Hence, counting incidences between simple vertices and facets, we have   
\[
 4n_4\leq12+2\cdot2=16,
\]
contrary to $n_4=6$ in \eqref{eq:fourteen-reduction}.

As a result, this subcase does not arise.

\textbf{Three stacked vertices, forming a triangle.} Let these vertices be $u,v,w$. 
Each edge  $\mathsf e$ of the triangle $u,v,w$ satisfies $f_3(\mathsf e,\mathsf P)=5$. 
At $u$, the two vertices corresponding to $[u,v]$ and $[u,w]$ are adjacent in $\mathsf P/u$ (these are vertices $v_1$ and $v_2$ in Figure~\ref{fig:fourteen-nonsimple-edges}), so these edges lie in a common ridge of $\mathsf P$. 
That ridge cannot be a quadrilateral, because $v$ and $w$ are adjacent. 
Thus $\mathsf R:=[u,v,w]$ is a triangular ridge.

Consider a facet $\mathsf F$ containing $\mathsf R$. 
At $u$, the corresponding triangular face of its vertex figure $\mathsf P/u$ contains an edge $\mathsf f$ formed by the degree-five vertices in $\mathsf P/u$. 
Each of the two degree-three vertices forms a triangular face of $\mathsf{P}/u$ with the two degree-five vertices; these are precisely the facets containing $\mathsf f$. 
Hence the triangular face of $\mathsf{P}/u$ corresponding to $\mathsf F$ contains a degree-three vertex. 
The same applies to $v$ and $w$. 
Consequently, all three edges of $\mathsf F$ joining $\mathsf R$ to the opposite triangular face $\mathsf R'$ of $\mathsf F$ are simple, and none of the incidences $(u,\mathsf F)$, $(v,\mathsf F)$ and $(w,\mathsf F)$ is special.
If one vertex of $\mathsf R'$ were simple, each of its neighbours in $\mathsf R'$ would be incident with at most one nonsimple edge of $\mathsf F$ and would therefore also be simple by Lemma~\ref{lem:fourteen-simple-facet}(i).
This would make all three vertices of $\mathsf R'$ simple, contrary to our assumption.
Hence both facets containing $\mathsf R$, say $\mathsf F$ and $\mathsf F'$, have no simple vertices.

The six simple vertices require $24$ vertex--facet incidences. 
The twelve facets  of $\mathsf P$ other than $\mathsf F$ and $\mathsf F'$ must therefore each contain exactly two simple vertices, since each facet in $\mathsf P$ contains at most two simple vertices (Lemma~\ref{lem:fourteen-simple-facet}(ii)). 
However, every stacked vertex has a special incidence with a facet not containing $\mathsf R$, and that facet has at most one simple vertex by Lemma~\ref{lem:fourteen-simple-facet}(ii). 
This is a contradiction that rules this subcase out.

As a consequence,  $\mathsf P$ must have a triangular ridge with three simple vertices. 
Lemma~\ref{lem:split} gives $\mathsf P\cong\mathsf B\cs\mathsf Q$, where $\mathsf B=\Delta^2\times\Delta^2$ and $f_3(\mathsf Q)=10$. 
By Theorem~\ref{thm:no-prime-ten}, $\mathsf Q\cong\mathsf B\cs\mathsf B$, completing the proof.
\end{proof}

\begin{remark}
\label{cor:fourteen-two-types}
Up to combinatorial isomorphism, there are exactly two
prism-facetted $4$-polytopes with fourteen facets. Write $\mathsf B=\Delta^2\times\Delta^2$ and let $S$ be the six vertices identified in forming $\mathsf Q:=\mathsf B\#\mathsf B$.  The facet of $\mathsf Q$ used for the third gluing contains three or four vertices of $S$, and these two classes are automorphism orbits. For the resulting polytope $\mathsf P$, put $n_i:=\#\set{u\in\Fc_0(\mathsf P):f_3(u,\mathsf P)=i}$. The resulting types are
distinguished by their numbers of vertices incident with four,
six and eight facets:
\(
 (n_4,n_6,n_8)=(6,6,3)\) and \(
 (n_4,n_6,n_8)=(7,4,4)\).
\end{remark}


\begin{thebibliography}{51}

\bibitem{Sch04}
A.~Schwartz, \emph{Constructions of cubical polytopes}, Ph.D.\ thesis,
  Technische Universit\"at Berlin, 2004.

\bibitem{SchZie04}
A.~Schwartz and G.~M. Ziegler, \emph{Construction techniques for cubical complexes,
  odd cubical $4$-polytopes, and prescribed dual manifolds},
  Experiment. Math. \textbf{13} (2004), no.~4, 385--413.

\bibitem{PerShe67}
M.~A. Perles and G.~C. Shephard, \emph{Facets and nonfacets of convex polytopes},
  Acta Math. \textbf{119} (1967), 113--145.
  
\bibitem{JosZie00}
M.~Joswig and G.~M. Ziegler, \emph{Neighborly cubical polytopes},
  Discrete Comput. Geom. \textbf{24} (2000), no.~2--3, 325--344.

\bibitem{BliBli98}
G.~Blind and R.~Blind, \emph{The almost simple cubical polytopes},
  Discrete Math. \textbf{184} (1998), 25--48.

\bibitem{Bar71}
D.~W. Barnette, \emph{The minimum number of vertices of a simple polytope},
  {Isr. J. Math.} \textbf{10} (1971), 121--125. 


\bibitem{Bar73}
D.~W. Barnette, \emph{A proof of the lower bound conjecture for convex
  polytopes}, {Pac. J. Math.} \textbf{46} (1973), 349--354. 

\bibitem{Bar69}
D.~W. Barnette, \emph{A simple $4$-dimensional nonfacet},
  Israel J. Math. \textbf{7} (1969), 16--20.





 \bibitem{Gru03}
 B.~Gr{{\"u}}nbaum, \emph{Convex polytopes}, 2nd ed., Graduate Texts in
   Mathematics, vol. 221, Springer-Verlag, New York, 2003, Prepared and with a
   preface by V. Kaibel, V. Klee and G. M. Ziegler.

\bibitem{McM71}
P.~McMullen, \emph{The minimum number of facets of a convex polytope}, J.
  London Math. Soc. (2) \textbf{3} (1971), 350--354. 

\bibitem{OpenAIChatGPT}
{OpenAI}, \emph{{ChatGPT} ({GPT-5.6 Sol})}, \url{https://chatgpt.com/}, 2026,
  Large language model, accessed 29 August 2026.

\bibitem{PfePilSan12}
J.~Pfeifle, V.~Pilaud, and F.~Santos, \emph{Polytopality and {C}artesian
  products of graphs}, {Isr. J. Math.} \textbf{192} (2012), no.~1, 121--141.


\bibitem{Pin24}
G.~Pineda-Villavicencio, \emph{Polytopes and graphs}, Cambridge Studies in
  Advanced Mathematics, vol. 211, Cambridge Univ. Press, Cambridge, 2024.

\bibitem{SAGE}
{The Sage Developers}, \emph{Sage mathematics software {(Version 10.9)}}, 2026.


\bibitem{AnthropicClaude}
{Anthropic}, \emph{{Claude (Opus 4.8)}}, \url{https://www.anthropic.com/claude/opus},
  2026, Large language model, accessed 29 August 2026.




  \bibitem{SanZie10}
R.~Sanyal and G.~M. Ziegler, \emph{Construction and analysis of projected
  deformed products}, Discrete Comput. Geom. \textbf{43} (2010), no.~2,
  412--435.



  

\end{thebibliography}
\end{document}